\documentclass[reqno]{amsproc}
\usepackage[utf8]{inputenc}
\usepackage{cite}
\usepackage{xcolor}
\usepackage{color}
\usepackage{calligra}
\usepackage[T1]{fontenc}
\usepackage{latexsym}
\usepackage[all]{xy}
\usepackage{hyperref}
\usepackage{mathrsfs}
\usepackage{amsmath,amsfonts,amssymb,amsxtra}
\numberwithin{equation}{section}
\newtheorem{theorem}{\bf Theorem}[section]

\newtheorem{lem}{\bf Lemma}[section]
\newtheorem{cor}{\bf Corollary}[section]
\newtheorem{remark}{\bf Remark}[section]

\newtheorem{example}{\bf Example}[section]

\newtheorem{question}{\bf Question}
\newcommand\dv{\mathrm{div}}
\newcommand\tr{\mathrm{tr}}
\usepackage{algorithmic}
\begin{document}

\title[Estimates of eigenvalues of an elliptic differential system]{Estimates of eigenvalues of an elliptic differential system in divergence form}
\author[M.C. Araújo Filho]{Marcio C. Araújo Filho$^1$}
\author[J.N.V. Gomes]{José N.V. Gomes$^2$}
\address{$^1$Departamento de Matemática, Universidade Federal de Rondônia, Campus Ji-Paraná, R. Rio Amazonas, 351, Jardim dos Migrantes, 76900-726 Ji-Paraná, Rondônia, Brazil}
\address{$^2$Departamento de Matemática, Universidade Federal de São Carlos, Rod. Washington Luíz, Km 235, 13565-905 São Carlos, São Paulo, Brazil}
\email{$^1$marcio.araujo@unir.br}
\email{$^2$jnvgomes@ufscar.br}

\begin{abstract}
In this paper, we compute universal estimates of eigenvalues of a coupled system of elliptic differential equations in divergence form on a bounded domain in Euclidean space. As an application, we show an interesting case of rigidity inequalities of the eigenvalues of the Laplacian, more precisely, we consider a countable family of bounded domains in Gaussian shrinking soliton that makes the behavior of known estimates of the eigenvalues of the Laplacian invariant by a first-order perturbation of the Laplacian. We also address the Gaussian expanding soliton case in two different settings. We finish with the special case of divergence-free tensors which is closely related to the Cheng-Yau operator.
\end{abstract}
\maketitle

\section{Introduction}
Let $\mathbb{R}^n$ be the $n$-dimensional Euclidean space with its canonical metric $\langle,\rangle$, and $\Omega\subset\mathbb{R}^n$ be a bounded domain with smooth boundary $\partial\Omega$. Let us consider a symmetric positive definite $(1,1)$-tensor  $T$ on $\mathbb{R}^n$ and a function $\eta\in C^2(\mathbb{R}^n)$, so that we can define a second-order elliptic differential operator $\mathscr{L}$ in the $(\eta,T)$-divergence form as follows:
\begin{equation}\label{eq1.1}
    \mathscr{L}f :=\dv_\eta (T(\nabla f)) = \dv(T(\nabla f)) - \langle \nabla \eta, T(\nabla f) \rangle,
\end{equation}
where $\dv$ stands for the divergence operator and $\nabla$ for the gradient operator. Since $\Omega$ is bounded, there exist two positive real constants $\varepsilon$ and $\delta$, such that $\varepsilon I\leq T\leq \delta I$, where $I$ is the $(1, 1)$-tensor identity on $\mathbb{R}^n$.

The analysis of the sequence of the eigenvalues of elliptic differential operators in divergence forms in bounded domains in $\mathbb{R}^n$ is an interesting topic in both mathematics and physics. In particular, problems linking the shape of a domain to the spectrum of an operator are among the most fascinating of mathematical analysis. One of the reasons which make them so attractive is that they involve diﬀerent ﬁelds of mathematics such as spectral theory, Riemannian geometry, and partial diﬀerential equations. Not only the literature about this subject is already very rich, but also it is not unlikely that operators in divergence forms may play a fundamental role in the understanding of countless physical facts.

In this paper, we address the eigenvalue problem for an operator which is a second-order perturbation of $\mathscr{L}$. More precisely, we compute universal estimates of the eigenvalues of the coupled system of second-order elliptic differential equations, namely:
\begin{equation}\label{problem1}
    \left\{\begin{array}{ccccc} 
    \mathscr{L}  {\bf u} + \alpha \nabla(\dv_\eta {\bf u}) &=& -\sigma {\bf u} & \mbox{in } & \Omega,\\
     {\bf u}&=&0 & \mbox{on} & \partial\Omega, 
    \end{array} 
    \right.
\end{equation}
where ${\bf u}=(u^1, u^2, \ldots, u^n)$ is a vector-valued function from $\Omega$ to $\mathbb{R}^n$, the constant $\alpha$ is non-negative and $\mathscr{L}{\bf u}=(\mathscr{L}u^1, \mathscr{L}u^2, \ldots,\mathscr{L} u^n)$. 

We will see that $\mathscr{L}+\alpha\nabla\dv_\eta$ is a formally self-adjoint operator in the Hilbert space $\mathbb{L}^2(\Omega,dm)$ of all vector-valued functions that vanish on $\partial \Omega$ in the sense of the trace. It follows from inner product induced by Eqs.~\eqref{parts} and \eqref{divformula} in Section~\ref{preliminar}. Thus the eigenvalue problem~\eqref{problem1} has a real and discrete spectrum 
\begin{equation}\label{equationn1.3}
    0 < \sigma_1 \leq \sigma_2 \leq \cdots \leq \sigma_k \leq \cdots\to\infty,
\end{equation}
where each $\sigma_i$ is repeated according to its multiplicity.

A special case that we can obtain from Problem~\ref{problem1} occurs when $T$ is divergence-free, see Problem~\ref{problem1-3}. For the sake of convenience, we address this case in Section~\ref{Sec-DCYOp}. Some results from Problems~\ref{problem3} and \ref{problem4} below are particular cases of this section. However, these two latter problems still remain prototype for us. In the next two paragraphs, we make brief comments about them.

When $\eta$ is a constant and $T$ is the identity operator $I$ on $\mathbb{R}^n$, Problem~\eqref{problem1} becomes
 \begin{equation}\label{problem3}
    \left\{ \begin{array}{ccccc} 
    \Delta  {\bf u} + \alpha \nabla(\dv\;{\bf u}) &=& -\sigma {\bf u} & \mbox{in} & \Omega,  \\
     {\bf u}&=& 0 & \mbox{on} & \partial \Omega,
    \end{array} 
    \right.
\end{equation}
where $\Delta{\bf u}=(\Delta u^1,\ldots,\Delta u^n)$ and $\Delta$ is the Laplacian operator on $C^\infty(\Omega)$. The operator $\Delta+\alpha\nabla \dv$ is known as Lamé's operator. In the $3$-dimensional case it shows up in the elasticity theory  and $\alpha$ is determined by the positive constants of Lamé, so the assumption $\alpha \geq 0$ is justified. For further details on this issue, the interested reader can consult Pleijel~\cite{Pleijel} or Kawohl and Sweers~\cite{KawohlSweers}. It is worth mentioning here the works of Levine and Protter~\cite{LevineProtter}, Livitin and Parnovski~\cite{LevitinParnovski}, Hook~\cite{Hook}, Cheng and Yang~\cite{ChengYang} and Chen et al.~\cite{CCWX} in which we can find some interesting estimates of the eigenvalues  of Problem~\eqref{problem3}. We will be more precise later when we will discuss the three latter papers.

When $\eta$ is not necessarily constant and $T=I$, Problem~\eqref{problem1} is rewritten as 
\begin{equation}\label{problem4}
    \left\{ \begin{array}{ccccc} 
    \Delta_\eta  {\bf u} + \alpha \nabla \mbox{(div}_\eta{\bf u}) &=& -\sigma {\bf u} & \mbox{in} &\Omega,  \\
     {\bf u}&=&0 & \mbox{on} & \partial \Omega,
    \end{array} 
    \right.
\end{equation}
where $\Delta_{\eta}{\bf u}=(\Delta_\eta u^1,\ldots,\Delta_\eta u^n)$ and $\Delta_\eta = \dv_\eta \nabla$ is the drifted Laplacian operator on $C^\infty(\Omega)$. The drifted Laplacian as well as the Bakry-Emery Ricci tensor $Ric+\nabla^2\eta$ are the most appropriate geometric objects to study the smooth metric measure spaces $(M^n,g,e^{-\eta}dvol_g)$. In particular, the Bakry-Emery Ricci tensor has been especially studied in the theory of Ricci solitons, since a gradient Ricci soliton $(M^n,g,\eta)$ is characterized by $Ric+\nabla^2\eta=\lambda g$, for some constant $\lambda$.

In Corollary~\ref{corollary2.4}, we show an interesting case of rigidity inequalities of eigenvalues of the Laplacian in a countable family of bounded domains in Gaussian shrinking soliton $(\mathbb{R}^n,\delta_{ij},\frac{\lambda}{2}|x|^2)$ by taking a specific isoparametric function as being the drifting function $\eta$, see Remarks~\ref{remark1-1} and \ref{remark-rigidity}. We address the Gaussian expanding soliton case in Corollaries~\ref{Expanding case1} and~\ref{Expanding case2}.

Throughout this paper $dm=e^{-\eta}d\Omega$ stands for the weight volume form on $\Omega$ and $|\cdot|$ for the Euclidean norm. Moreover, let us define
\begin{eqnarray*}
\nabla {\bf u} = ( \nabla u^1, \ldots ,  \nabla u^n) \quad \mbox{and} \quad T(\nabla{\bf u}) = (T(\nabla u^1), \ldots , T(\nabla u^n)),
\end{eqnarray*}
so that
\begin{eqnarray}\label{norm-Tu}
\|T (\nabla {\bf u})\|^2 = \int_\Omega \sum_{j=1}^n |T(\nabla u^j)|^2 dm = \int_\Omega |T(\nabla {\bf u})|^2 dm.
\end{eqnarray}

Henceforth, since there is no danger of confusion, we are using the same notation $\|\cdot\|$ for the norm in \eqref{norm-Tu} as well as for the canonical norm of a real-valued function in $L^2(\Omega,dm)$.

Our proofs will be facilitated by analyzing the more general setting in which the function $\eta$ is not necessarily constant and $T$ is not necessarily the identity. In this case, we prove a universal quadratic estimate for the eigenvalues of Problem~\eqref{problem1}, which is an essential tool to obtain some of our estimates.
\begin{theorem}\label{theorem1.1}
Let $\Omega \subset \mathbb{R}^n$ be a  bounded domain and ${\bf u}_i$ be a normalized eigenfunction corresponding to $i$-th eigenvalue $\sigma_i$ of Problem~\eqref{problem1}. Then, for any positive integer $k$, we get
\begin{equation*}
    \sum_{i=1}^k(\sigma_{k+1}-\sigma_i)^2 \leq \frac{4\delta(n\delta+\alpha)}{n^2\varepsilon^2}\sum_{i=1}^k(\sigma_{k+1}-\sigma_i)\Big\{\Big[(\sigma_i - \alpha \|\dv_\eta{\bf u}_i\|^2)^{\frac{1}{2}}  + \frac{T_0}{2\sqrt{\delta}} \Big]^2 + \frac{C_0}{\delta}\Big\},
\end{equation*}
where
\begin{equation}\label{C_0}
    C_0=\sup_{\Omega}{\Big \{}\frac{1}{2}\dv (T^2(\nabla \eta)) - \frac{1}{4}|T(\nabla \eta)|^2{\Big \}} + \frac{\delta}{2} T_0\eta_0,
\end{equation}
$T_0=\sup_{\Omega}|\tr(\nabla T)|$ and $\eta_0=\sup_{\Omega}|\nabla \eta|$.
\end{theorem} 

\begin{remark} Notice that the constant $C_0$ in Eq.~\eqref{C_0} has been appropriately defined such that $\Big[(\sigma_i - \alpha \|\dv_\eta{\bf u}_i\|^2)^{\frac{1}{2}}  + \frac{T_0}{2\sqrt{\delta}} \Big]^2 + \frac{C_0}{\delta}>0$, for $i=1,\dots,k$.
\end{remark}

We identify the quadratic estimate in Theorem~\ref{theorem1.1} as the most appropriate inequality for the applications of our results. In particular, the constant $C_0$ in \eqref{C_0} has a crucial importance for us.

Theorem~\ref{theorem1.1} is an extension for $\mathscr{L}+\alpha\nabla\mbox{div}_\eta$ on vector-valued functions of the well-known Yang's estimate of the eigenvalues of the Laplacian on real-valued functions. Its proof is motivated by the corresponding results for the Laplacian on real-valued functions case by Yang~\cite{Yang}, for $\Delta+\alpha\nabla \dv$ on vector-valued functions case by Chen et al.~\cite[Theorem~1.1]{CCWX}, and for $\mathscr{L}$ on real-valued functions case by Gomes and Miranda~\cite{GomesMiranda}. 

We also prove an estimate for the sum of lower order eigenvalues in terms of the first eigenvalue and its correspondent eigenfunction. 

\begin{theorem}\label{theorem1.2}
Let $\Omega \subset \mathbb{R}^n$ be a bounded domain, $\sigma_i$ be the $i$-th eigenvalue of Problem~\eqref{problem1}, for $i=1,\ldots,n$, and ${\bf u}_1$ be a normalized eigenfunction corresponding to the first eigenvalue. Then, we get
\begin{equation*}
    \sum_{i=1}^n(\sigma_{i+1} - \sigma_1) \leq \frac{4\delta(\delta+\alpha)}{\varepsilon^2}\Big\{\Big[(\sigma_1 - \alpha \|\dv_\eta{\bf u}_1\|^2)^{\frac{1}{2}}  + \frac{T_0 }{2\sqrt{\delta}}\Big]^2 + \frac{C_0}{\delta}\Big\}.
\end{equation*}
\end{theorem}

Theorem~\ref{theorem1.2} is an extension for $\mathscr{L}+\alpha\nabla\mbox{div}_\eta$ on vector-valued functions of a stronger result obtained by Cheng and Yang~\cite[Theorem~1.2]{ChengYang} for lower order eigenvalues of Problem~\eqref{problem3}. Its proof is motivated by the corresponding results  for $\Delta+\alpha\nabla \dv$ case in \cite{ChengYang} and for $\Delta_\eta+\alpha\nabla \dv_\eta$ case in~\cite[Theorem~1.3]{DuBezerra}.

\section{Applications}\label{section}
We begin this section by defining a known class of the functions which is closely related to our applications. A nonconstant smooth function $f: \mathbb{R}^n \to \mathbb{R}$ is called transnormal function if
\begin{equation}\label{Eq-transnormal}
    |\nabla f|^2 = b(f),
\end{equation}
for some smooth function $b$ on the range of $f$ in $\mathbb{R}$. The function $f$ is called an isoparametric function if it moreover satisfies
\begin{equation}\label{Eq-isoparametric}
    \Delta f = a(f),
\end{equation}
for some continuous function  $a$ on the range of $f$ in $\mathbb{R}$. 

Eq.~\eqref{Eq-transnormal} implies that the level set hypersurfaces of $f$ are parallel hypersurfaces and it follows from Eq.~\eqref{Eq-isoparametric} that these hypersurfaces have constant mean curvature. Isoparametric functions appear in the isoparametric hypersurfaces theory (i.e., has constant principal curvatures) systematically developed by Cartan~\cite{Cartan} on space forms. Wang~\cite{Wang} considered the problem of extending this theory to a general Riemannian manifold and studied some properties of \eqref{Eq-transnormal} and \eqref{Eq-isoparametric} more closely. Notice that isoparametric functions exist on a large class of spaces (e.g. symmetric spaces) other than space forms. Currently, new examples of isoparametric functions on Riemannian manifolds have been discovered, for instance, the potential function of any noncompact gradient Ricci soliton $(M^n,g,\eta)$ with constant scalar curvature $R$ is an isoparametric function, since we can assume that $\eta$ (after a possible rescaling) satisfies $|\nabla \eta|^2=2\lambda\eta-R$ and $\Delta\eta=\lambda n-R$, see, e.g., Chow et al.~\cite{B.Chow}. In particular, the potential function of the Gaussian shrinking soliton $(\mathbb{R}^n,\delta_{ij},\frac{\lambda}{2}|x|^2)$ is an isoparametric function, see Example~\ref{example1}. This latter fact and a brief analysis of the constant $C_0$ in \eqref{C_0} were the main motivations to consider the isoparametric function $\eta(x)=\frac{\lambda}{2}|x|^2$ to give some applications of our results. The quadratic estimate below is a basic result for it.

\begin{cor}\label{theorem3.1}
Let $\Omega\subset\mathbb{R}^n$ be a bounded domain and ${\bf u}_i$ be a normalized eigenfunction corresponding to $i$-th eigenvalue $\sigma_i$ of Problem~\eqref{problem4}. Then, for any positive integer $k$, we get
\begin{align}\label{eq3.2}
    \sum_{i=1}^k(\sigma_{k+1}-\sigma_i)^2 &\leq \frac{4( n+\alpha)}{n^2}\sum_{i=1}^k(\sigma_{k+1}-\sigma_i)( \sigma_i -\alpha\|\dv_\eta {\bf u}_i\|^2 + C_0),
\end{align} 
where $C_0 = \sup_{\Omega}\big\{ \frac{1}{2} \Delta \eta - \frac{1}{4}|\nabla \eta |^2\big\}$. Moreover, $\sigma_i -\alpha\|\dv_\eta {\bf u}_i\|^2 + C_0>0$, for $i=1,\ldots,k$.
\end{cor}
\begin{proof}
In Problem~\eqref{problem4} we must have  $T = I$. Then, we get $\varepsilon=\delta =1$ and  $T_0 = 0$. Hence, the result of the corollary follows from Theorem~\ref{theorem1.1}. 
\end{proof}

The following corollary is an immediate consequence of Theorem~\ref{theorem1.2}.
\begin{cor}\label{new-corollary}
Let $\Omega\subset\mathbb{R}^n$ be a  bounded domain, $\sigma_i$ be the $i$-th eigenvalue of Problem~\eqref{problem4}, for $i=1,\ldots,n$, and ${\bf u}_1$ be a normalized eigenfunction corresponding to the first eigenvalue. Then, we get
\begin{align}\label{equation5.4}
    \sum_{i=1}^n(\sigma_{i+1} - \sigma_1) \leq 4(1+\alpha)( \sigma_1 + D_1),
\end{align}
where $D_1 = -\alpha \|\dv_\eta {\bf u}_1\|^2+ C_0$ and $C_0 = \sup_{\Omega}\Big\{ \frac{1}{2} \Delta \eta - \frac{1}{4}|\nabla \eta |^2\Big\}$.
\end{cor}

Notice that the appearance of the constant $C_0$ is natural, since we did not impose any restriction on the function $\eta$. We highlight that this constant has an unexpected geometric interpretation. Indeed, let us consider the warped metric $g = g_0 + e^{-\eta}d\theta^2$ on the product $\Omega\times \mathbb{S}^1$, where $g_0$ stands for the canonical metric in the domain $\Omega\subset\mathbb{R}^n$, whereas $d\theta^2$ is the canonical metric of the unit sphere $\mathbb{S}^1$, so that the scalar curvature of $g$ is given by $\frac{1}{2} \Delta \eta - \frac{1}{4}|\nabla \eta |^2$. Hence, we can obtain $C_0$ as supremum of the scalar curvature on the warped product $\Omega\times_{e^{-\eta}}\mathbb{S}^1$. Moreover, we ask the following natural question:
\begin{question}
Under which conditions the inequalities for the eigenvalues obtained from~\eqref{eq3.2} and \eqref{equation5.4} do not depend on the constant $C_0$ for a nontrivial function $\eta$?
\end{question}

We give an answer to this question by using a specific family of domains in Gaussian shrinking soliton. More precisely, we consider a countable family of bounded domains in $\mathbb{R}^n$ that makes the behavior of known estimates of eigenvalues of the Laplacian invariant by a first-order perturbation of the Laplacian, see  Corollary~\ref{corollary2.4}. 

Coming back to Corollary~\ref{theorem3.1}, we define
\begin{equation}\label{Do and D1}
D_0=-\alpha \min_{j=1, \ldots, k}\|\dv_\eta {\bf u}_j\|^2+C_0,
\end{equation}
so that, from~\eqref{eq3.2}, we get
\begin{align}\label{eq2.5} 
    \sum_{i=1}^k(\sigma_{k+1}-\sigma_i)^2 &\leq \frac{4( n+\alpha)}{n^2}\sum_{i=1}^k(\sigma_{k+1}-\sigma_i)( \sigma_i + D_0).
\end{align}
Notice that $\sigma_i + D_0>0.$

Now, as mentioned in the introduction, we immediately recover the following inequality:
\begin{align}\label{Ineq-CCWX}
   \sum_{i=1}^k(\sigma_{k+1}-\sigma_i)^2 \leq \frac{4(n+\alpha)}{ n^2}\sum_{i=1}^k(\sigma_{k+1}-\sigma_i)\sigma_i,
  \end{align}
which has been obtained by Chen et al.~\cite[Corollary~1.2]{CCWX} for Problem~\eqref{problem3}. Indeed, it follows from \eqref{Do and D1} and \eqref{eq2.5}, since $\alpha\geq0$ and we can take $\eta$ to be a constant. Moreover, Inequality~\eqref{Ineq-CCWX} implies Theorem~1.1 in Cheng and Yang~\cite{ChengYang}, whereas \cite[Theorem~1.1]{ChengYang} implies Theorem~10 in Hook~\cite{Hook}.  However, we highlight that Inequality~\eqref{eq2.5} provides an estimate for the eigenvalues of Problem~\eqref{problem3} which is better than Inequality~\eqref{Ineq-CCWX}. 

In the case of Problem~\eqref{problem4}, we can see that Inequality~\eqref{equation5.4} is better than Inequality~(1.7) in Du and Bezerra~\cite{DuBezerra}; whereas Inequality~\eqref{eq2.5} is better than Inequality~(1.3) again in \cite{DuBezerra}.

Besides, from Inequality~\eqref{eq2.5} and following the steps of the proof of \cite[Theorem~3]{GomesMiranda}, we obtain the inequalities:
\begin{cor}\label{cor_sharpgap}
Under the same setup as in Corollary~\ref{theorem3.1}, we have
\begin{align}\label{new-sharpgap}
    \sigma_{k+1} + D_0 \leq& \Big(1+\frac{2(n+\alpha)}{n^2}\Big)\frac{1}{k}\sum_{i=1}^k(\sigma_i + D_0) + \Big[\Big( \frac{2(n+\alpha)}{n^2}\frac{1}{k}\sum_{i=1}^k(\sigma_i + D_0)\Big)^2 \nonumber \\
   &-\Big(1+\frac{4(n+\alpha)}{n^2}\Big) \frac{1}{k}\sum_{j=1}^k\Big(\sigma_j-\frac{1}{k}\sum_{i=1}^k\sigma_i\Big)^2 \Big]^{\frac{1}{2}}
\end{align}
and
\begin{align}\label{new-sharpgap2}
\nonumber\sigma_{k+1} -\sigma_k \leq& 2\Big[\Big( \frac{2(n+\alpha)}{n^2}\frac{1}{k}\sum_{i=1}^k(\sigma_i + D_0)\Big)^2\\ &-\Big(1+\frac{4(n+\alpha)}{n^2}\Big) \frac{1}{k}\sum_{j=1}^k\Big(\sigma_j-\frac{1}{k}\sum_{i=1}^k\sigma_i\Big)^2 \Big]^{\frac{1}{2}},
\end{align}
where $D_0$ is given by \eqref{Do and D1}.
\end{cor}

\begin{proof}
By setting $v_i := \sigma_i + D_0$ into \eqref{eq2.5} we get
\begin{equation}\label{equation1.13}
    \sum_{i=1}^k(v_{k+1}-v_i)^2 \leq \frac{4(n+\alpha)}{n^2}\sum_{i=1}^k(v_{k+1}-v_i)v_i. 
\end{equation}
From~\eqref{equationn1.3}, $v_1 \leq v_2 \leq \cdots \leq v_{k+1}$, and from~\eqref{eq2.5} each $v_i>0$. Notice that \eqref{equation1.13} is a quadratic inequality of $v_{k+1}$. So, we can follow the steps of the proof of Inequalities~(1.3) and (1.4) in Gomes and Miranda \cite{GomesMiranda} to obtain the required inequalities of the corollary.
\end{proof}

Again from \eqref{eq2.5} and by applying the recursion formula of Cheng and Yang~\cite{ChengYang2}, we obtain the following corollary.
\begin{cor}\label{corollary3.3}
Under the same setup as in Corollary~\ref{theorem3.1}, we have
\begin{equation}\label{equation1.11}
    \sigma_{k+1} + D_0 \leq \Big(1+\frac{4(n+\alpha)}{n^2}\Big)k^{\frac{2( n+\alpha)}{n^2}}(\sigma_1 + D_0),
\end{equation}
where $D_0$ is given by \eqref{Do and D1}.
\end{cor}

\begin{proof}
Notice that the recursion formula by Cheng and Yang~\cite[Corollary~2.1]{ChengYang2} remains true for any positive real number, in particular, it holds for $\frac{n^2}{n+\alpha}$, then we can apply this formula in \eqref{equation1.13} to obtain \eqref{equation1.11}.
\end{proof}

From the classical Weyl’s asymptotic formula for the eigenvalues \cite{Weyl}, we know that estimate \eqref{equation1.11} is optimal in the sense of the order on $k$. 

\begin{remark}\label{remark1}
If $D_0 = 0$, then the inequalities of eigenvalues \eqref{eq2.5} and \eqref{equation1.11} have the same behavior as the known estimates of the eigenvalues of $\Delta + \alpha \nabla \dv$, see Inequality~\eqref{Ineq-CCWX} and Chen et al.~\cite[Corollary~1.4]{CCWX}, respectively. In the same way, the inequalities of eigenvalues  \eqref{new-sharpgap} and \eqref{new-sharpgap2} imply 
\begin{equation*}
    \sigma_{k+1} \leq \Big(1+\frac{4(n + \alpha)}{n^2}\Big)\frac{1}{k}\sum_{i=1}^k\sigma_i \quad \mbox{and} \quad
     \sigma_{k+1} - \sigma_k \leq \frac{4(n + \alpha)}{n^2}\frac{1}{k}\sum_{i=1}^k\sigma_i,
 \end{equation*}
which have the same behavior as the inequalities of the eigenvalue of $\Delta + \alpha \nabla \dv$ obtained by Chen et al.~\cite[Corollary~1.3]{CCWX}.

If $D_1 = 0$, then the inequality of eigenvalues~\eqref{equation5.4} has the same behavior as the known estimate of the eigenvalues of $\Delta + \alpha \nabla \dv$ proved by Cheng and Yang~\cite[Theorem~1.2]{ChengYang}.
\end{remark}

\begin{remark}\label{remark1-1}
For $\alpha=0$ case, if $C_0=0$ for some function $\eta$ (possibly radial or isoparametric), then the inequalities  \eqref{equation5.4}, \eqref{eq2.5}, \eqref{new-sharpgap} and \eqref{equation1.11} have the same behavior as the known estimates of the eigenvalues of the Laplacian, see \cite[Inequality~(6.2)]{AshbaughBenguria}, \cite[Theorem~1]{Yang},   \cite[Inequality~(1.8)]{ChengYang2} and \cite[Corollary~2.1]{ChengYang2}, respectively.
\end{remark}

Example~\ref{example1} below is a special case of $C_0=0$. To see this, let us consider an isoparametric function  $\eta(x)=\frac{\lambda}{2}(x_1^2+\cdots+x_k^2)$ on $\mathbb{R}^n$, where $\lambda$ is any nonzero real number, $k$ an integer with $0<k\leq n$ and $x=(x_1, \ldots, x_n)\in \mathbb{R}^n$. It is easy to verify that
\begin{equation*}
    |\nabla \eta|^2=2\lambda\eta  \:\: \mbox{and} \:\: \Delta \eta = \lambda k.
\end{equation*}
In particular, if $k=n$, the function $\eta(x)=\frac{\lambda}{2}|x|^2$ is the potential function of the Gaussian shrinking ($\lambda>0$) or expanding ($\lambda<0$) soliton on $\mathbb{R}^n$. We now take $\eta(x)=\frac{\lambda}{2}|x|^2$ into the equation of $C_0$ in Corollary~\ref{theorem3.1}, so that,
\begin{equation}\label{C_0-Ex}
    C_0=\sup_{\Omega}\Big\{\frac{\lambda n}{2} - \frac{\lambda^2}{4}|x|^2 \Big\}.
\end{equation}

With these considerations in mind, we write the next two examples.
\begin{example}\label{example1}
Let us consider the family of bounded domains $\{\Omega_l\}_{l=1}^\infty$ in Gaussian shrinking or expanding soliton $(\mathbb{R}^n,\delta_{ij},\frac{\lambda}{2}|x|^2)$ given by
\begin{equation*}\label{domains}
        \Omega_l=\mathbb{B}(r_l)-\bar{\mathbb{B}}(\sqrt{2n/|\lambda|})=\Big\{ x \in \mathbb{R}^n; \frac{2n}{|\lambda|} < |x|^2 < r_l^2 \Big\},
\end{equation*}
where $r_l>\sqrt{2n/|\lambda|}$ is a rational number, and $\mathbb{B}(r)$ stands for the open ball of radius $r$ centered at the origin in $\mathbb{R}^n$. So,
\begin{equation*}
    \min_{\bar{\Omega}_l}|x|^2=\frac{2n}{|\lambda|}, \quad \mbox{for all} \quad l=1, 2, \ldots.
\end{equation*}
\noindent{\bf(a) Shrinking case:} 
\begin{equation*}
    C_0=\frac{\lambda}{2}\sup_{\Omega_l}\Big\{n - \frac{\lambda}{2}|x|^2 \Big\} = \frac{\lambda}{2}\Big(n - \frac{\lambda}{2}\min_{\bar{\Omega}_l}|x|^2\Big)=0, \quad \mbox{for all} \quad l=1, 2, \ldots.
\end{equation*}
\noindent{\bf(b) Expanding case:}
\begin{equation*}
    C_0=\frac{\lambda}{2}\inf_{\Omega_l}\Big\{n - \frac{\lambda}{2}|x|^2 \Big\} = \frac{\lambda}{2}\Big(n - \frac{\lambda}{2}\min_{\bar{\Omega}_l}|x|^2\Big)= \lambda n, \quad \mbox{for all} \quad l=1, 2, \ldots.
\end{equation*}
\end{example}

\begin{example}\label{example1-1}
Let us consider the domain $\Omega$ to be the open ball $\mathbb{B}(r)$ of radius $r$ centered at the origin in Gaussian shrinking or expanding soliton $(\mathbb{R}^n,\delta_{ij},\frac{\lambda}{2}|x|^2)$. From Eq.~\eqref{C_0-Ex}, we easily see that $C_0 = \lambda n/2$ for both shrinking and expanding case.
\end{example}

We are now in the right position to give the interesting applications that we had promised. 

\begin{cor}[{\bf Non-dependence of $\eta$}]\label{corollary2.4}
Let us consider the family of domains $\{\Omega_l\}_{l=1}^\infty$ given by Example~\ref{example1} in Gaussian shrinking soliton $(\mathbb{R}^n,\delta_{ij},\frac{\lambda}{2}|x|^2)$. Let $\sigma_i$ be the $i$-th eigenvalue of the drifted Laplacian $\Delta_\eta$ on real-valued functions, with drifting function $\eta(x)=\frac{\lambda}{2}|x|^2$, on each $\Omega_l$ with Dirichlet boundary condition. Then, we get
\begin{align*}
    \sum_{i=1}^k(\sigma_{k+1}-\sigma_i)^2 &\leq \frac{4}{n}\sum_{i=1}^k(\sigma_{k+1}-\sigma_i)\sigma_i,
\end{align*} 
\begin{align*}
    \sigma_{k+1} \leq \Big(1+\frac{2}{n}\Big)\frac{1}{k}\sum_{i=1}^k\sigma_i + \Big[\Big( \frac{2}{n}\frac{1}{k}\sum_{i=1}^k\sigma_i\Big)^2 -\Big(1+\frac{4}{n}\Big) \frac{1}{k}\sum_{j=1}^k\Big(\sigma_j-\frac{1}{k}\sum_{i=1}^k\sigma_i\Big)^2 \Big]^{\frac{1}{2}},
\end{align*}
\begin{align*}
\nonumber\sigma_{k+1} -\sigma_k \leq 2\Big[\Big( \frac{2}{n}\frac{1}{k}\sum_{i=1}^k\sigma_i\Big)^2 -\Big(1+\frac{4}{n}\Big) \frac{1}{k}\sum_{j=1}^k\Big(\sigma_j-\frac{1}{k}\sum_{i=1}^k\sigma_i\Big)^2 \Big]^{\frac{1}{2}},
\end{align*}
\begin{equation*}
    \sigma_{k+1} \leq \Big(1+\frac{4}{n}\Big)k^{\frac{2}{n}}\sigma_1
\end{equation*}
and
\begin{equation*}
    \sum_{i=1}^n(\sigma_{i+1} - \sigma_1) \leq 4\sigma_1.
\end{equation*}
\end{cor}
\begin{proof}
We start by taking $\alpha=0$ as in Problem~\eqref{problem4}. Next, we note that the constant $C_0=0$ for the shrinking case. So, the required inequalities follow from inequalities \eqref{eq2.5}, \eqref{new-sharpgap}, \eqref{new-sharpgap2}, \eqref{equation1.11} and \eqref{equation5.4}, respectively.
\end{proof}

\begin{remark}\label{remark-rigidity}
Notice that Corollary~\ref{corollary2.4} can be regarded as rigidity inequalities (see Remark~\ref{remark1-1}) on the family of bounded domains $\{\Omega_l\}_{l=1}^\infty$ given by Example~\ref{example1} in Gaussian shrinking soliton $(\mathbb{R}^n,\delta_{ij},\frac{\lambda}{2}|x|^2)$.
\end{remark}

Now, we will address the expanding case. 
\begin{cor}\label{Expanding case1}
Let $\mathbb{B}(r)$ be the open ball of radius $r$ centered at the origin in Gaussian expanding soliton $(\mathbb{R}^n,\delta_{ij},\frac{\lambda}{2}|x|^2)$. Let $\sigma_1$ be the first eigenvalue of the drifted Laplacian $\Delta_\eta$ on real-valued functions, with drifting function $\eta(x)=\frac{\lambda}{2}|x|^2$, on $\mathbb{B}(r)$ with Dirichlet boundary condition. Then, we have
\begin{equation*}
\sigma_1\geq \frac{\pi^2n}{64r^2}-\frac{\lambda n}{2},
\end{equation*}
and the next estimate for the sum of lower order eigenvalues $\sigma_i$ of $\Delta_\eta$ in terms of the first eigenvalue 
\begin{equation}\label{Gap22}
\sum_{i=1}^n(\sigma_{i+1}-\sigma_1)\leq 4(\sigma_1+\frac{\lambda n}{2}).
\end{equation}
\end{cor}

\begin{proof}
Let $\sigma_1$ and $\sigma_2$ be the first and second eigenvalues of the drifted Laplacian $\Delta_\eta$ on real-valued functions on $\mathbb{B}(r)$ with Dirichlet boundary condition, respectively. If $\lambda<0$, then both $\eta=\frac{\lambda}{2}|x|^2$ and $f= \frac{1}{2} \Delta \eta - \frac{1}{4}|\nabla \eta|^2$ are concave functions on the closure of the convex domain $\mathbb{B}(r)$. Thus, we can apply Theorem~3 by Ma and Liu~\cite{Ma-Liu} to obtain 
\begin{equation}\label{Gap1}
\sigma_2-\sigma_1\geq \frac{\pi^2}{16r^2}.
\end{equation}
On the other hand, note that we can use \eqref{new-sharpgap2} or \eqref{equation1.11}, for $\alpha=0$, to get 
\begin{equation}\label{Gap2}
\sigma_2-\sigma_1\leq \frac{4}{n}\sigma_1+2\lambda.
\end{equation}
Combining \eqref{Gap1} and \eqref{Gap2}, we conclude that
\begin{equation*}
\sigma_1\geq \frac{\pi^2n}{64r^2}-\frac{\lambda n}{2}.
\end{equation*}
Moreover, we can use \eqref{equation5.4}, for $\alpha=0$, to obtain \eqref{Gap22}.
\end{proof}

\begin{cor}\label{Expanding case2}
Let us consider the family of domains $\{\Omega_l\}_{l=1}^\infty$ given by Example~\ref{example1} in Gaussian expanding soliton $(\mathbb{R}^n,\delta_{ij},\frac{\lambda}{2}|x|^2)$. Let $\sigma_i$ be the $i$-th eigenvalue of the drifted Laplacian $\Delta_\eta$ on real-valued functions, with drifting function $\eta(x)=\frac{\lambda}{2}|x|^2$, on each  $\Omega_l$ with Dirichlet boundary condition. Then, it is valid the following estimate for the sum of lower order eigenvalues of $\Delta_\eta$ in terms of the first eigenvalue:
\begin{equation*}
\sum_{i=1}^n(\sigma_{i+1}-\sigma_1)\leq 4(\sigma_1+\lambda n).
\end{equation*}
\end{cor}
\begin{proof}
In fact, we can use \eqref{equation5.4}, for $\alpha=0$, to deduce the required estimate.
\end{proof}

\begin{remark}
A final remark is in order. We observe that Corollaries~\ref{theorem3.1} and~\ref{new-corollary} can be obtained from Corollaries~\ref{corollary-6.1} and \ref{corollary-6.2}, respectively. Whereas Corollaries~\ref{cor_sharpgap} and \ref{corollary3.3} can be obtained from Corollaries~\ref{corollary-6.3} and \ref{corollary3.3-CY}, respectively. However, as we already mentioned before, they have been a prototype for us.
\end{remark}

\section{Preliminaries}\label{preliminar}

This section is brief and serves to set the stage, introducing some basic notation and describing what is meant by the properties of a $(1,1)$-tensor in a bounded domain $\Omega\subset\mathbb{R}^n$ with smooth boundary $\partial\Omega$.

Throughout the paper, we will be constantly using the identification of a $(0,2)$-tensor $T:\mathfrak{X}(\Omega)\times\mathfrak{X}(\Omega)\to C^{\infty}(\Omega)$ with its associated $(1,1)$-tensor $T:\mathfrak{X}(\Omega)\to\mathfrak{X}(\Omega)$ by the equation 
\begin{equation*}
    \langle T(X), Y \rangle = T(X, Y).
\end{equation*}
In particular, the tensor $\langle , \rangle$ will be identified with the identity $I$ in $\mathfrak{X}(\Omega)$. From the definition of $\eta$-divergence of $X$ (see Eq.~\eqref{eq1.1}) and the usual properties of divergence of vector fields, one has
\begin{equation}\label{property1}
    \dv_\eta(fX)=f\dv_\eta X + \nabla f \cdot X ,
\end{equation}
for all $f \in C^\infty(\Omega)$.

Notice that the $(\eta,T)$-divergence form of $\mathscr{L}$ on $\Omega$ allows us to check that the divergence theorem remains true in the form
\begin{equation}\label{property2}
   \int_\Omega\dv_\eta X dm = \int_{\partial\Omega} \langle X, \nu\rangle d\mu.
\end{equation}
In particular, for $X=T(\nabla f)$, 
\begin{equation*}
   \int_\Omega\mathscr{L}{f}dm = \int_{\partial\Omega}T(\nabla f, \nu) d\mu, 
\end{equation*}
where $dm = e^{-\eta}d\Omega$ and $d\mu = e^{-\eta}d\partial\Omega$ are the weight volume form on $\Omega$ and the volume form on the boundary $\partial\Omega$ induced by the outward unit normal vector $\nu$ on $\partial \Omega$, respectively. Thus, the integration by parts formula is given by
\begin{equation}\label{parts}
     \int_{\Omega}\ell\mathscr{L}{f}dm =-\int_\Omega T(\nabla\ell, \nabla f)dm + \int_{\partial\Omega}\ell T(\nabla f, \nu) d\mu,
\end{equation}
for all $\ell, f \in C^\infty(\Omega)$. Hence, $\mathscr{L}$ is a formally self-adjoint operator in the space of all real-valued functions in $L^2(\Omega, dm)$ that vanish on $\partial \Omega$ in the sense of the trace. Furthermore, from \eqref{property1} and \eqref{property2} we obtain
\begin{equation}\label{divformula}
    \int_\Omega {\bf v} \cdot \nabla(\dv_\eta {\bf u})dm = -\int_\Omega \dv_\eta {\bf u} \dv_\eta {\bf v} dm, 
\end{equation}
for all  vector-valued function ${\bf u}=(u^1, u^2, \ldots, u^n)$ and ${\bf v}=(v^1, v^2, \ldots, v^n)$ both from $\Omega$ to $\mathbb{R}^n$, with ${\bf v}$ vanishing on $\partial \Omega$.

We conclude that $\mathscr{L}+\alpha\nabla\dv_\eta$ is a formally self-adjoint operator in the Hilbert space $\mathbb{L}^2(\Omega,dm)$ of all vector-valued functions that vanish on $\partial \Omega$ in the sense of the trace, with inner product induced by Eqs.~\eqref{parts} and \eqref{divformula}.  Thus the eigenvalue problem~\eqref{problem1} has a real and discrete spectrum 
$0 < \sigma_1 \leq \sigma_2 \leq \cdots \leq \sigma_k \leq \cdots\to\infty$,
where each $\sigma_i$ is repeated according to its multiplicity.

It is worth mentioning here the paper by Gomes and Miranda~\cite[Section~2]{GomesMiranda} from which we know some geometric motivations to work with the operator $\mathscr{L}$ in the $(\eta,T)$-divergence form in bounded domains in Riemannian manifolds. They showed that it appears as the trace of a $(1,1)$-tensor on a Riemannian manifold $M$, and computed a Bochner-type formula for it. An interesting fact is that Eq.~ $(2.3)$ of \cite{GomesMiranda} relates the operator $\mathscr{L}$ to Cheng and Yau's operator $\Box$ defined in~\cite{ChengYau}. In particular, if we take $T$ to be the Einstein tensor, then this latter relation is likely to have applications in physics.

For a vector-valued function ${\bf u}=(u^1, u^2, \ldots, u^n)$ from $\Omega$ to $\mathbb{R}^n$, we define 
\begin{eqnarray*}
\nabla {\bf u} = ( \nabla u^1, \ldots ,  \nabla u^n).
\end{eqnarray*}

Now, we are considering two definitions for $T$, and since there is no danger of confusion, we are using the same notation $T$ for both definitions.
Let $X$ be a vector field on $\Omega$, we define 
\begin{eqnarray*}
T(\nabla{\bf u}) = (T(\nabla u^1), \ldots , T(\nabla u^n))
\end{eqnarray*}
and
\begin{align}\label{definition}
    T(X, \nabla {\bf u}) &= (\langle T(X), \nabla { u}^1 \rangle, \ldots, \langle T(X), \nabla { u}^n \rangle)\nonumber\\
    &=(\langle X, T(\nabla u^1) \rangle, \ldots, \langle X, T(\nabla u^n) \rangle)\nonumber\\
     &=(T(X, \nabla u^1), \ldots, T(X, \nabla u^n)).
\end{align}
Moreover, for all real-valued functions $f, \ell \in C^\infty (\Omega)$ it is immediate from the properties of $\dv_\eta$ and the symmetry of $T$ that
\begin{equation*}
    \mathscr{L}(f\ell) = f\mathscr{L}\ell + 2 T(\nabla f, \nabla\ell) + \ell\mathscr{L}f.
\end{equation*}
So, the following equation is well understood for a vector-valued function ${\bf u}$ and a real-valued function $f \in C^\infty (\Omega)$
\begin{align}\label{equation2.2}
    &\mathscr{L}(f {\bf u}) = (\mathscr{L}(f u^1), \ldots, \mathscr{L}(f u^n))\nonumber\\
    &= (f\mathscr{L}u^1+2 \langle T(\nabla f), \nabla u^1\rangle+ \mathscr{L}(f) u^1, \ldots, f\mathscr{L}u^n+2\langle T(\nabla f), \nabla u^n\rangle+ \mathscr{L}(f) u^n) \nonumber\\
    &=f(\mathscr{L}u^1, \ldots, \mathscr{L}u^n) + 2(\langle T(\nabla f), \nabla u^n\rangle, \ldots, \langle T(\nabla f), \nabla u^n\rangle) + \mathscr{L}(f)( u^1, \ldots, u^n)\nonumber\\
    &=f\mathscr{L}{\bf u} + 2T(\nabla f, \nabla {\bf u}) + \mathscr{L}f{\bf u}.
\end{align}
Besides, we are using the classical norms: $|{\bf u}|^2 = \sum_{i=1}^n(u^i)^2$ and $\|{\bf u}\|^2= \int_\Omega |{\bf u}|^2 dm$. We observe that $\varepsilon I\leq T\leq \delta I$ implies
\begin{align}\label{T-property}
|T(X)|^2 \leq \delta  \langle T(X), X\rangle \quad \mbox{for all} \quad X \in \mathfrak{X}(\mathbb{R}^n).
\end{align}
In particular, we obtain
\begin{align}\label{T-norm}
    |T(\nabla\eta)|^2 \leq \delta^2|\nabla\eta|^2.
\end{align}

\section{Three technical lemmas}
In order to prove our first theorem we will need three technical lemmas. The first one is motivated by the corresponding results for Problem~\eqref{problem3} proven by Chen et al.~\cite[Lemma~2.1]{CCWX} and for Problem~\eqref{problem4} proven by Du and Bezerra~\cite[Lemma~2.1]{DuBezerra}. Here, we follow the steps of the proof of Lemma~2.1 in \cite{CCWX} with appropriate adaptations for $\mathscr{L}+\alpha\nabla \dv_\eta$.
\begin{lem}\label{lema1}
Let $\Omega\subset\mathbb{R}^n$ be a bounded domain, $\sigma_i$ be the i-th eigenvalue of Problem~\eqref{problem1} and ${\bf u}_i$ be a normalized vector-valued eigenfunction corresponding to $\sigma_i$. Then, for any $f\in C^2(\Omega)\cap C^1(\partial \Omega)$ and any positive constant $B$, we obtain
\begin{align*}
    \sum_{i=1}^k(\sigma_{k+1}-\sigma_i)^2 &\Big\{(1-B)\int_{\Omega}T(\nabla f, \nabla f)|{\bf u}_i|^2dm -B\alpha \int_\Omega|\nabla f\cdot{\bf u}_i|^2dm\Big\}\\
    &\leq \frac{1}{B}\sum_{i=1}^k(\sigma_{k+1}-\sigma_i)\|T(\nabla f, \nabla {\bf u}_i) + \frac{1}{2}\mathscr{L}f{\bf u}_i\|^2.
\end{align*}
\end{lem}

\begin{proof}
Let ${\bf u}_i$ be a normalized vector-valued eigenfunction corresponding to $\sigma_i$, i.e., it satisfies
\begin{equation}\label{eq4.1}
    \left\{ \begin{array}{ccccc} 
    \mathscr{L}{\bf u}_i + \alpha \nabla(\dv_\eta{\bf u}_i) &=& -\sigma_i  {\bf u}_i & \mbox{in} & \Omega,  \\
      {\bf u}_i &=& 0 & \mbox{on} & \partial \Omega,  \\ 
     \int_\Omega  {\bf u}_i \cdot {\bf u}_j dm &=& \delta_{ij} & \mbox{for any} & i, j.
    \end{array}
    \right.
\end{equation}
Since $\sigma_{k+1}$ is the minimum value of the Rayleigh quotient (see, e.g., \cite[Theorem~9.43]{PJ-Olver}), we must have 
\begin{equation}\label{rayleigh-1}
    \sigma_{k+1} \leq - \frac{\int_\Omega {\bf v}\cdot(\mathscr{L}{\bf v}+\alpha \nabla(\dv_\eta {\bf v}))dm}{\int_\Omega |{\bf v}|^2dm},
\end{equation}
for any nonzero vector-valued function ${\bf v}: \Omega \to \mathbb{R}^n$ satisfying 
\begin{equation*}
    {\bf v}|_{\partial \Omega}=0 \quad \mbox{and} \quad \int_\Omega  {\bf v} \cdot{\bf u}_i dm = 0, \quad \mbox{for any}, \quad i=1,\ldots, k.
\end{equation*}
Let us denote by $a_{ij}=\int_\Omega f  {\bf u}_i\cdot {\bf u}_j dm = a_{ji}$ to consider the vector-valued functions ${\bf v}_i$ given by
\begin{equation}\label{vi}
    {\bf v}_i=f{\bf u}_i - \sum_{j=1}^ka_{ij}{\bf u}_j,
\end{equation}
so that 
\begin{equation}\label{eq7}
{\bf v}_i|_{\partial \Omega}=0 \quad \mbox{and}\quad \int_\Omega{\bf u}_j\cdot{\bf v}_i dm = 0, \quad \mbox{for any}\quad i,j=1,\ldots, k.
\end{equation}
Then, we can take ${\bf v}={\bf v}_i$ in \eqref{rayleigh-1} and use formula~\eqref{divformula} to obtain
\begin{equation}\label{rayleigh}
\sigma_{k+1}\|{\bf v}_i\|^2 \leq \int_\Omega \Big(-{\bf v}_i\cdot\mathscr{L}{\bf v}_i + \alpha(\dv_\eta{\bf v}_i)^2\Big)dm.
\end{equation}
From \eqref{vi} and \eqref{equation2.2}, we get 
\begin{align*}
    \mathscr{L}{\bf v}_i=&f\mathscr{L}{\bf u}_i + 2T(\nabla f, \nabla{\bf u}_i) + \mathscr{L}f{\bf u}_i-\sum_{j=1}^ka_{ij}\mathscr{L}{\bf u}_j\\
    =&f(-\sigma_i{\bf u}_i-\alpha \nabla(\dv_\eta{\bf u}_i))+2T(\nabla f, \nabla {\bf u}_i) + \mathscr{L}f{\bf u}_i\\
    &-\sum_{j=1}^ka_{ij}(-\sigma_j{\bf u}_j-\alpha \nabla(\dv_\eta {\bf u}_j)),
\end{align*}
i.e.,
\begin{align*}
    \mathscr{L}{\bf v}_i =&-\sigma_if{\bf u}_i+\sum_{j=1}^ka_{ij}\sigma_j{\bf u}_j+2T(\nabla f, \nabla{\bf u}_i)+\mathscr{L}f{\bf u}_i\\
   &-\alpha f\nabla(\dv_\eta {\bf u}_i)+\alpha\sum_{j=1}^ka_{ij}\nabla(\dv_\eta {\bf u}_j).
\end{align*}
Therefore,
\begin{align}\label{eq9}
    \int_\Omega -{\bf v}_i \cdot \mathscr{L}{\bf v}_i dm =&  \sigma_i\|{\bf v}_i\|^2 - \int_\Omega {\bf v}_i \cdot (2T(\nabla f, \nabla {\bf u}_i)
   +\mathscr{L}f{\bf u}_i)dm\nonumber \\
   &+\alpha \Big(\int_\Omega f {\bf v}_i\cdot\nabla ( \dv_\eta{\bf u}_i)dm - \sum_{j=1}^ka_{ij}\int_\Omega {\bf v_i}\cdot \nabla(\dv_\eta {\bf u}_j) dm \Big).
\end{align}
From \eqref{property1} and \eqref{divformula}
\begin{equation*}
   \int_\Omega f {\bf v}_i\cdot\nabla(\dv_\eta {\bf u}_i) dm = -\int_\Omega f \dv_\eta {\bf u}_i \dv_\eta {\bf v}_i dm  -\int_\Omega  \dv_\eta {\bf u}_i \nabla f \cdot {\bf v}_i dm .
\end{equation*}
But, from \eqref{vi}
\begin{equation*}
    -f \dv_\eta {\bf u}_i = - \dv_\eta {\bf v}_i + \nabla f \cdot {\bf u}_i- \sum_{j=1}^ka_{ij}\dv_\eta {\bf u}_j,
\end{equation*}
then
\begin{align*}
   \int_\Omega f {\bf v}_i\cdot\nabla(\dv_\eta {\bf u}_i) dm =& -\int_\Omega  (\dv_\eta {\bf v}_i)^2 dm  + \int_\Omega  \dv_\eta {\bf v}_i \nabla f \cdot {\bf u}_i dm \\
   &- \sum_{j=1}^ka_{ij}\int_\Omega \dv_\eta {\bf u}_j\dv_\eta{\bf v}_i dm -\int_\Omega  \dv_\eta {\bf u}_i \nabla f \cdot {\bf v}_i dm\\
   =& -\int_\Omega  (\dv_\eta {\bf v}_i)^2 dm  + \int_\Omega  \dv_\eta {\bf v}_i \nabla f \cdot {\bf u}_i dm \\
   &+ \sum_{j=1}^ka_{ij}\int_\Omega {\bf v}_i\cdot \nabla(\dv_\eta {\bf u}_j) dm -\int_\Omega  \dv_\eta {\bf u}_i \nabla f \cdot {\bf v}_i dm.
\end{align*}
Thus, 
\begin{align}\label{eq4.7}
   &\int_\Omega f {\bf v}_i\cdot\nabla(\dv_\eta {\bf u}_i) dm - \sum_{j=1}^ka_{ij}\int_\Omega {\bf v}_i\cdot\nabla(\dv_\eta {\bf u}_j) dm \nonumber\\
    &= -\int_\Omega (\dv_\eta {\bf v}_i)^2 dm + \int_\Omega (\dv_\eta {\bf v}_i\nabla f \cdot {\bf u}_i - \dv_\eta {\bf u}_i \nabla f \cdot {\bf v}_i)dm \nonumber\\
    &=-\int_\Omega (\dv_\eta {\bf v}_i)^2dm -\int_\Omega (\nabla(\nabla f \cdot {\bf u}_i)+\dv_\eta {\bf u}_i \nabla f)\cdot {\bf v}_i dm.
\end{align}
So, replacing \eqref{eq4.7} into \eqref{eq9}, we obtain
\begin{align}\label{eq9_2}
    -\sigma_i\|{\bf v}_i\|^2=& \int_\Omega {\bf v}_i \cdot \mathscr{L}{\bf v}_i dm  - \int_\Omega {\bf v}_i \cdot (2T(\nabla f, \nabla {\bf u}_i)
   +\mathscr{L}f{\bf u}_i)dm\nonumber \\
   &-\alpha\int_\Omega (\dv_\eta {\bf v}_i)^2dm -\alpha\int_\Omega (\nabla(\nabla f \cdot {\bf u}_i)+\dv_\eta {\bf u}_i \nabla f)\cdot {\bf v}_i dm.
\end{align}
Hence, from \eqref{rayleigh} and \eqref{eq9_2}, we have
\begin{align}\label{eq10}
    (\sigma_{k+1}-\sigma_i)\|{\bf v}_i\|^2 \leq& -\int_\Omega (2T(\nabla f, \nabla{\bf u}_i)
   +\mathscr{L}f{\bf u}_i)\cdot {\bf v}_idm \nonumber\\
   &-\alpha\int_\Omega (\nabla(\nabla f\cdot {\bf u}_i)+\dv_\eta {\bf u}_i\nabla f)\cdot {\bf v}_idm.
\end{align}
Using integration by parts formula~\eqref{parts} and \eqref{vi}, we get
\begin{equation}\label{eq13}
    \int_\Omega (2T(\nabla f, \nabla {\bf u}_i)
   +\mathscr{L}f{\bf u}_i)\cdot {\bf v}_idm = -\int_\Omega |{\bf u}_i|^2T(\nabla f,\nabla f)dm - 2\sum_{j=1}^ka_{ij}b_{ij},
\end{equation}
where
\begin{equation}\label{bij}
    b_{ij}=\int_\Omega {\Big (}T(\nabla f, \nabla {\bf u}_i)
   +\frac{1}{2}\mathscr{L}f{\bf u}_i{\Big )}\cdot{\bf u}_jdm=-b_{ji}.
\end{equation}
Moreover, by straightforward computation from \eqref{property1}, \eqref{property2} and \eqref{vi}, we have
\begin{align}\label{eq14}
    &\int_\Omega (\nabla(\nabla f \cdot {\bf u}_i)+\dv_\eta {\bf u}_i \nabla f)\cdot {\bf v}_idm\nonumber\\
    &=\sum_{j=1}^ka_{ij}\int_\Omega(\nabla f\cdot {\bf u}_i\dv_\eta {\bf u}_j - \dv_\eta {\bf u}_i \nabla f\cdot{\bf u}_j)dm - \int_\Omega |\nabla f \cdot {\bf u}_i|^2dm.
\end{align}
Putting 
\begin{equation}\label{w_i}
    w_i=-\int_\Omega (2T(\nabla f, \nabla {\bf u}_i)
   +\mathscr{L}f{\bf u}_i)\cdot {\bf v}_idm  -\alpha\int_\Omega {\Big (}\nabla(\nabla f \cdot {\bf u}_i)+\dv_\eta {\bf u}_i \nabla f{\Big )}\cdot {\bf v}_idm,
\end{equation}
from \eqref{eq13} and \eqref{eq14}
\begin{align*}
    w_i=&\int_\Omega|{\bf u}_i|^2T(\nabla f, \nabla f)dm +2\sum_{j=1}^ka_{ij}b_{ij} \nonumber\\
    &- \alpha \sum_{j=1}^ka_{ij}\int_\Omega(\nabla f \cdot {\bf u}_i\dv_\eta {\bf u}_j- \dv_\eta {\bf u}_i \nabla f \cdot {\bf u}_j)dm +\alpha \int_\Omega |\nabla f \cdot {\bf u}_i|^2dm.
\end{align*}
By a similar computation as in~\cite[Eq. ~(2.9)]{CCWX}, from \eqref{eq4.1} and \eqref{bij}, we get
\begin{equation*}
    2b_{ij}=(\sigma_i-\sigma_j)a_{ij} +\alpha \int_\Omega (\nabla f\cdot {\bf u}_i\dv_\eta {\bf u}_j- \dv_\eta {\bf u}_i\nabla f \cdot {\bf u}_j) dm,
\end{equation*}
then
\begin{equation*}
    2\sum_{j=1}^ka_{ij} b_{ij}=\sum_{j=1}^k(\sigma_i-\sigma_j)a_{ij}^2 +\alpha\sum_{j=1}^k a_{ij}\int_\Omega (\nabla f \cdot {\bf u}_i\dv_\eta {\bf u}_j - \dv_\eta {\bf u}_i \nabla f \cdot {\bf u}_j)dm.
\end{equation*}
Thus,
\begin{align}\label{eq15}
    w_i&=\int_\Omega|{\bf u}_i|^2T(\nabla f, \nabla f)dm + \sum_{j=1}^k(\sigma_i-\sigma_j)a_{ij}^2
    +\alpha \int_\Omega |\nabla f \cdot {\bf u}_i|^2dm.
\end{align}
Furthermore, from \eqref{eq10} and \eqref{w_i}, we have
\begin{equation}\label{eq16}
    (\sigma_{k+1}-\sigma_i)\|{\bf v}_i\|^2 \leq w_i.
\end{equation}
For any constant $B>0$, from \eqref{eq7}, \eqref{eq13} and the inequality of Cauchy-Schwarz, we infer
\begin{align*}
    (\sigma_{k+1}&-\sigma_i)^2\Big( \int_\Omega T(\nabla f, \nabla f)|{\bf u}_i|^2dm + 2\sum_{j=1}^ka_{ij}b_{ij}\Big)\\
    =&(\sigma_{k+1}-\sigma_i)^2\Bigg\{ -2\int_\Omega \Big(T(\nabla f, \nabla {\bf u}_i)+\frac{1}{2}\mathscr{L}f{\bf u}_i - \sum_{j=1}^kb_{ij}{\bf u}_j\Big)\cdot {\bf v}_idm \Bigg\} \\
    \leq& 2(\sigma_{k+1}-\sigma_i)^2\|{\bf v}_i\|\Big\|T(\nabla f, \nabla {\bf u}_i)+\frac{1}{2}\mathscr{L}f{\bf u}_i - \sum_{j=1}^kb_{ij}{\bf u}_j\Big\|\\
    \leq& (\sigma_{k+1}-\sigma_i)^3B\|{\bf v}_i\|^2 + \frac{\sigma_{k+1}-\sigma_i}{B}\Big\|T(\nabla f, \nabla {\bf u}_i )+\frac{1}{2}\mathscr{L}f{\bf u}_i - \sum_{j=1}^kb_{ij}{\bf u}_j\Big\|^2,
\end{align*}
hence, using \eqref{eq15} and \eqref{eq16}, we obtain
\begin{align*}
    (\sigma_{k+1}&-\sigma_i)^2\Big( \int_\Omega T(\nabla f, \nabla f)|{\bf u}_i|^2dm + 2\sum_{j=1}^ka_{ij}b_{ij}\Big)\\
    \leq& (\sigma_{k+1}-\sigma_i)^2B\Big(\int_\Omega|{\bf u}_i|^2T(\nabla f, \nabla f)dm + \sum_{j=1}^k(\sigma_i-\sigma_j)a_{ij}^2
    +\alpha \int_\Omega |\nabla f \cdot {\bf u}_i|^2dm\Big)\\
      &+\frac{\sigma_{k+1}-\sigma_i}{B} \Big( \Big\|T(\nabla f, \nabla {\bf u}_i)+\frac{1}{2}\mathscr{L}f{\bf u}_i\Big\|^2 - \sum_{j=1}^kb_{ij}^2 \Big).
\end{align*}
Summing over $i$ from $1$ to $k$, we obtain
\begin{align}\label{eq19}
    \sum_{i=1}^k&(\sigma_{k+1}-\sigma_i)^2\Big( \int_\Omega T(\nabla f, \nabla f)|{\bf u}_i|^2dm + 2\sum_{j=1}^ka_{ij}b_{ij}\Big) \nonumber\\
    \leq& \sum_{i=1}^k(\sigma_{k+1}-\sigma_i)^2B\Big(\int_\Omega|{\bf u}_i|^2T(\nabla f, \nabla f)dm + \sum_{j=1}^k(\sigma_i-\sigma_j)a_{ij}^2
    +\alpha \int_\Omega |\nabla f\cdot {\bf u}_i|^2dm\Big) \nonumber\\
      &+\sum_{i=1}^k\frac{\sigma_{k+1}-\sigma_i}{B}\Big(\Big\|T(\nabla f, \nabla{\bf u}_i)+\frac{1}{2}\mathscr{L}f{\bf u}_i\Big\|^2 - \sum_{j=1}^kb_{ij}^2\Big).
\end{align}
Since $a_{ij}=a_{ji}$ and $b_{ij}=-b_{ji}$, we have
\begin{align*}
    2\sum_{i,j=1}^k(\sigma_{k+1}-\sigma_i)^2a_{ij}b_{ij}=-2\sum_{i,j=1}^k(\sigma_{k+1}-\sigma_i)(\sigma_i-\sigma_j)a_{ij}b_{ij},
\end{align*}
and
\begin{align*}
    \sum_{i,j=1}^k(\sigma_{k+1}-\sigma_i)^2(\sigma_i-\sigma_j)a_{ij}^2=-\sum_{i,j=1}^k(\sigma_{k+1}-\sigma_i)(\sigma_i-\sigma_j)^2a_{ij}^2.
\end{align*}
Therefore, from \eqref{eq19}, we get
\begin{align*}
   \sum_{i=1}^k &(\sigma_{k+1}-\sigma_i)^2 \int_\Omega T(\nabla f, \nabla f)|{\bf u}_i|^2dm \nonumber\\
   \leq & \sum_{i=1}^k(\sigma_{k+1}-\sigma_i)^2B\Big(\int_\Omega|{\bf u}_i|^2T(\nabla f, \nabla f)dm +\alpha \int_\Omega |\nabla f\cdot {\bf u}_i|^2dm\Big) \nonumber\\
      &+\sum_{i=1}^k\frac{\sigma_{k+1}-\sigma_i}{B}\Big\|T(\nabla f, \nabla {\bf u}_i)+\frac{1}{2}\mathscr{L}f{\bf u}_i\Big\|^2,
\end{align*}
whence
\begin{align*}
    &\sum_{i=1}^k(\sigma_{k+1}-\sigma_i)^2 \Big( (1-B)\int_\Omega T(\nabla f, \nabla f)|{\bf u}_i|^2dm-B\alpha \int_\Omega |\nabla f \cdot {\bf u}_i|^2dm \Big) \nonumber\\
    &\leq \frac{1}{B}\sum_{i=1}^k{(\sigma_{k+1}-\sigma_i)}\Big\|T(\nabla f, \nabla {\bf u}_i)+\frac{1}{2}\mathscr{L}f{\bf u}_i\Big\|^2.
\end{align*}
This finishes the proof of Lemma \ref{lema1}.
\end{proof}
The proof of the next lemma follows the steps of the proof of Proposition~2 in Gomes and Miranda~\cite{GomesMiranda} with appropriate adaptations for vector-valued functions from $\Omega$ to $\mathbb{R}^n$.
\begin{lem}\label{lema2}
Let $\Omega\subset\mathbb{R}^n$ be a bounded domain, $\sigma_i$ be the $i$-th eigenvalue of Problem~\eqref{problem1} and ${\bf u}_i$ be a normalized vector-valued eigenfunction corresponding to $\sigma_i$. Then, for any positive integer $k$, we get
\begin{align*}
     \sum_{i=1}^k(\sigma_{k+1}-&\sigma_i)^2 \leq \frac{4(n\delta+\alpha)}{n^2\varepsilon^2} \sum_{i=1}^k(\sigma_{k+1}-\sigma_i)\Big\{\frac{1}{4}\int_{\Omega}|{\bf u}_i|^2|\tr(\nabla T)-T(\nabla \eta)|^2dm \nonumber\\
    &+ \int_\Omega {\bf u}_i\cdot{\Big (} T(\tr(\nabla T), \nabla {\bf u}_i) - T(T(\nabla \eta), \nabla {\bf u}_i){\Big )}dm + \|T(\nabla {\bf u}_i)\|^2\Big\}.
\end{align*}
\end{lem}
\begin{proof} Let $\{x_\beta\}_{\beta=1}^n$ be the coordinate functions of $\mathbb{R}^n$, then taking $f=x_\beta$ in Lemma~\ref{lema1} and summing over $\beta$ from 1 to $n$, we get
\begin{align*}
    \sum_{i=1}^k(\sigma_{k+1}-&\sigma_i)^2\sum_{\beta=1}^n\Bigg\{ (1-B)\int_{\Omega}T(\nabla x_\beta, \nabla x_\beta)|{\bf u}_i|^2dm -B\alpha \int_{\Omega}|\nabla x_\beta \cdot {\bf u}_i|^2dm \Bigg\}\nonumber\\
   \leq& \frac{1}{B}\sum_{i=1}^k(\sigma_{k+1}-\sigma_i)\sum_{\beta=1}^n\Big\|T(\nabla x_\beta, \nabla {\bf u}_i) + \frac{1}{2}\mathscr{L}x_\beta {\bf u}_i\Big\|^2\nonumber\\
   =& \frac{1}{B}\sum_{i=1}^k(\sigma_{k+1}-\sigma_i)\int_\Omega\sum_{\beta=1}^n\Big|T(\nabla x_\beta, \nabla {\bf u}_i) + \frac{1}{2}\dv_\eta (T(\nabla x_\beta )){\bf u}_i\Big|^2dm.
\end{align*}
Therefore, 
\begin{align}\label{eq4.17}
    \sum_{i=1}^k(\sigma_{k+1}-&\sigma_i)^2\sum_{\beta=1}^n\Bigg\{ (1-B)\int_{\Omega}T(\nabla x_\beta, \nabla x_\beta)|{\bf u}_i|^2dm -B\alpha \int_{\Omega}|\nabla x_\beta \cdot {\bf u}_i|^2dm \Bigg\}\nonumber\\
   \leq& \frac{1}{B}\sum_{i=1}^k(\sigma_{k+1}-\sigma_i)\int_{\Omega}\sum_{\beta=1}^n\Bigg\{ \frac{|{\bf u}_i|^2}{4}(\dv_\eta (T(\nabla x_\beta)))^2\nonumber\\
   & + {\bf u}_i \cdot \Big(\dv_\eta (T(\nabla x_\beta))T(\nabla x_\beta, \nabla {\bf u}_i)\Big) + |T(\nabla x_\beta, \nabla {\bf u}_i)|^2  {\Bigg \}}dm.
\end{align}
By straightforward computation, we have
\begin{align*}
      \quad \quad \sum_{\beta=1}^nT(\nabla x_\beta, \nabla x_\beta)=\sum_{\beta=1}^n\langle T(e_\beta), e_\beta \rangle = \tr(T) \quad \mbox{and} \quad \sum_{\beta=1}^n|\nabla x_\beta \cdot {\bf u}_i|^2=|{\bf u}_i|^2.
\end{align*} 
Similarly to the calculations in~\cite[Eq.~(3.23)]{GomesMiranda}, we obtain
\begin{align*}
    \sum_{\beta=1}^n (\dv_\eta (T(\nabla x_\beta)))^2=|\tr(\nabla T) - T(\nabla \eta)|^2.
\end{align*}
From  ~\cite[Eq.~(3.24)]{GomesMiranda}, for all $k= 1, \ldots, n$, we get
\begin{equation*}
    \sum_{\beta=1}^n\dv_\eta (T(\nabla x_\beta))T(\nabla x_\beta, \nabla u_i^k) = \langle \tr(\nabla T), T(\nabla u_i^k)\rangle - \langle T(\nabla \eta), T(\nabla u_i^k) \rangle, 
\end{equation*}
then, using \eqref{definition}, we obtain
\begin{align*}
    &\sum_{\beta=1}^n\dv_\eta (T(\nabla x_\beta)) T(\nabla x_\beta, \nabla {\bf u}_i)=\\ 
    &=\Big(\sum_{\beta=1}^n\dv_\eta (T(\nabla x_\beta))T(\nabla x_\beta, \nabla u_i^1) , \ldots, \sum_{\beta=1}^n\dv_\eta (T(\nabla x_\beta))T(\nabla x_\beta, \nabla u_i^n)\Big)\\
    &=(\langle \tr(\nabla T), T(\nabla u_i^1)\rangle - \langle T(\nabla \eta), T(\nabla u_i^1) \rangle, \ldots, \langle \tr(\nabla T), T(\nabla u_i^n)\rangle - \langle T(\nabla \eta), T(\nabla u_i^n) \rangle)\\
    &=(\langle \tr(\nabla T), T(\nabla u_i^1)\rangle, \ldots, \langle \tr(\nabla T), T(\nabla u_i^n)\rangle) - (\langle T(\nabla \eta), T(\nabla u_i^1) \rangle, \ldots, \langle T(\nabla \eta), T(\nabla u_i^n) \rangle)\\
    &=T(\tr(\nabla T), \nabla {\bf u}_i) - T(T(\nabla \eta), \nabla {\bf u}_i).
\end{align*}
Moreover,
\begin{align*}
    \sum_{\beta=1}^n&|T(\nabla x_\beta, \nabla {\bf u}_i)|^2= \sum_{\beta=1}^n | T(e_\beta, \nabla {\bf u}_i)|^2 = \sum_{\beta=1}^n  |\big(\langle e_\beta, T(\nabla u_i^1)\rangle, \ldots, \langle e_\beta, T(\nabla u_i^n)\rangle \big)|^2 \\
    &=\sum_{\beta=1}^n \langle e_\beta, T(\nabla u_i^1)\rangle^2 + \cdots + \sum_{\beta=1}^n \langle e_\beta, T(\nabla u_i^n)\rangle^2
    =\sum_{j=1}^n|T(\nabla  u_i^j)|^2 = |T(\nabla {\bf u}_i)|^2.
\end{align*}
Substituting the previous equalities into \eqref{eq4.17} and remembering that $\|T (\nabla {\bf u}_i)\|^2 = \int_\Omega |T(\nabla {\bf u}_i)|^2 dm$, we get
\begin{align}\label{eq4.18}
   \sum_{i=1}^k& (\sigma_{k+1}-\sigma_i)^2 \Big[(1-B)\int_\Omega \tr(T)|{\bf u}_i|^2dm - B\alpha\Big]\nonumber\\
   \leq &\frac{1}{B}\sum_{i=1}^k(\sigma_{k+1}-\sigma_i){\Bigg \{}\frac{1}{4}\int_{\Omega}|{\bf u}_i|^2|\tr(\nabla T)-T(\nabla \eta)|^2 dm\nonumber\\
   &+ \int_\Omega{\bf u}_i\cdot{\Big (}T(\tr(\nabla T), \nabla {\bf u}_i) - T(T(\nabla \eta), \nabla {\bf u}_i){\Big )}dm + \|T (\nabla {\bf u}_i)\|^2  {\Bigg \}}.
\end{align}
Since $\varepsilon I \leq T \leq \delta I$, then $n\varepsilon \leq \tr(T) \leq n\delta$. Hence from \eqref{eq4.18} 
\begin{align}\label{equation4.18}
   \sum_{i=1}^k& (\sigma_{k+1}-\sigma_i)^2 [n\varepsilon - (n\delta+\alpha)B]\nonumber\\
   \leq &\frac{1}{B}\sum_{i=1}^k(\sigma_{k+1}-\sigma_i){\Bigg \{}\frac{1}{4}\int_{\Omega}|{\bf u}_i|^2|\tr(\nabla T)-T(\nabla \eta)|^2 dm\nonumber\\
   &+ \int_\Omega{\bf u}_i\cdot{\Big (}T(\tr(\nabla T), \nabla {\bf u}_i) - T(T(\nabla \eta), \nabla {\bf u}_i){\Big )}dm + \|T (\nabla {\bf u}_i)\|^2  {\Bigg \}}.
\end{align}
Furthermore, since $B$ is arbitrary positive constant, putting 
\begin{align*}
    B=&{\Big \{}\frac{\sum_{i=1}^k(\sigma_{k+1}-\sigma_i)}{{(n\delta+\alpha)\sum_{i=1}^k(\sigma_{k+1}-\sigma_i)^2}}{\Big [} \frac{1}{4}\int_{\Omega}|{\bf u}_i|^2|\tr(\nabla T)-T(\nabla \eta)|^2 dm   \nonumber\\
    &+ \int_\Omega{\bf u}_i\cdot (T(\tr(\nabla T), \nabla {\bf u}_i) - T(T(\nabla \eta), \nabla {\bf u}_i))dm + \|T (\nabla {\bf u}_i)\|^2{\Big ]}{\Big \}}^{\frac{1}{2}}
\end{align*}
into \eqref{equation4.18}, we obtain the required inequality and complete the proof of Lemma~\ref{lema2}.
\end{proof}

With these considerations in mind, we can rewrite the previous lemma in a more convenient way for us.
\begin{lem}\label{lemma2.3}
Under the same setup as in Lemma~\ref{lema2}, we get
\begin{align*}
     \sum_{i=1}^k&(\sigma_{k+1}-\sigma_i)^2 \leq \frac{4(n\delta+\alpha)}{n^2\varepsilon^2} \sum_{i=1}^k(\sigma_{k+1}-\sigma_i)\Big\{\|T(\nabla {\bf u}_i)\|^2 + C \\  
     &+\frac{1}{4}\int_{\Omega}|{\bf u}_i|^2\langle \tr(\nabla T), \tr(\nabla T) - 2T(\nabla \eta) \rangle dm +\int_{\Omega}{\bf u}_i \cdot T(\tr(\nabla T), \nabla {\bf u}_i) dm
   \Big\},
\end{align*}
where $C=\sup_{\Omega}\big\{\frac{1}{2}\dv (T^2(\nabla \eta)) - \frac{1}{4}|T(\nabla \eta)|^2\big\}$ has been chosen such that the term on the right-hand side must be positive.
\end{lem}
\begin{proof} We make use of Lemma~\ref{lema2}. For this, we must notice that
\begin{align*}
    |\tr(\nabla T)-T(\nabla \eta)|^2 = |\tr(\nabla T)|^2 - 2\langle \tr(\nabla T), T(\nabla \eta) \rangle + |T(\nabla \eta)|^2,
\end{align*}
hence
\begin{align}\label{eqq2.23}
     \frac{1}{4}&\int_{\Omega}|{\bf u}_i|^2|\tr(\nabla T)-T(\nabla \eta)|^2 dm + \int_\Omega{\bf u}_i \cdot (T(\tr(\nabla T), \nabla {\bf u}_i) -  T(T(\nabla \eta), \nabla {\bf u}_i)))dm \nonumber \\
     =&\frac{1}{4}\int_{\Omega}|{\bf u}_i|^2|T(\nabla \eta)|^2 dm -\int_\Omega{\bf u}_i \cdot T( T(\nabla \eta), \nabla {\bf u}_i) dm  +\frac{1}{4}\int_{\Omega}|{\bf u}_i|^2|\tr(\nabla T)|^2dm\nonumber\\
     &- \frac{1}{2}\int_\Omega|{\bf u}_i|^2\langle \tr(\nabla T), T (\nabla \eta)\rangle dm +\int_\Omega{\bf u}_i \cdot T(\tr(\nabla T), \nabla {\bf u}_i)dm.
\end{align}
Since ${\bf u}_i|_{\partial \Omega}=0$ by Eq.~\eqref{definition} and the divergence theorem, we have
\begin{align*}
    -&\int_\Omega{\bf u}_i \cdot T(T(\nabla \eta),\nabla {\bf u}_i) dm\\
    &= -\int_\Omega u_i^1 \langle T^2(\nabla \eta), \nabla u_i^1\rangle dm - \cdots  -\int_\Omega u_i^n \langle T^2(\nabla \eta), \nabla u_i^n\rangle dm \\
   &= -\frac{1}{2}\int_\Omega \langle T^2(\nabla \eta), \nabla (u_i^1)^2\rangle dm - \cdots  -\frac{1}{2}\int_\Omega \langle T^2(\nabla \eta), \nabla (u_i^n)^2\rangle dm \\
    &=\frac{1}{2}\int_\Omega (u_i^1)^2 \dv_\eta (T^2(\nabla \eta))dm + \cdots + \frac{1}{2}\int_\Omega (u_i^n)^2 \dv_\eta (T^2(\nabla \eta))dm \\
    &=\frac{1}{2}\int_\Omega |{\bf u}_i|^2 \dv_\eta (T^2(\nabla \eta))dm.
\end{align*}
Substituting the previous equation in Eq. \eqref{eqq2.23}, we get
\begin{align*}
     \frac{1}{4}&\int_{\Omega}|{\bf u}_i|^2|\tr(\nabla T)-T(\nabla \eta)|^2 dm + \int_\Omega{\bf u}_i\cdot (T(\tr(\nabla T), \nabla {\bf u}_i) - T( T(\nabla \eta), \nabla {\bf u}_i))dm \nonumber \\
     =&\int_{\Omega}|{\bf u}_i|^2{\Big(}\frac{1}{4}|T(\nabla \eta)|^2 +\frac{1}{2}\dv_\eta (T^2(\nabla \eta)){\Big)}dm   +\frac{1}{4}\int_{\Omega}|{\bf u}_i|^2\langle \tr(\nabla T), \tr(\nabla T) \rangle dm\nonumber\\
     &- \frac{1}{2}\int_\Omega|{\bf u}_i|^2\langle \tr(\nabla T), T (\nabla \eta)\rangle dm +\int_\Omega{\bf u}_i \cdot T(\tr(\nabla T), \nabla {\bf u}_i) dm.
\end{align*}
By setting 
\begin{equation*}
    C=\sup_{\Omega}\Big\{\frac{1}{4}|T(\nabla \eta)|^2 +\frac{1}{2}\dv_\eta (T^2(\nabla \eta))\Big\} = \sup_{\Omega}\Big\{\frac{1}{2}\dv (T^2(\nabla \eta)) - \frac{1}{4}|T(\nabla \eta)|^2\Big\}
\end{equation*}
and by the previous equality, we have
\begin{align}\label{eqq2.24}
     \frac{1}{4}&\int_{\Omega}|{\bf u}_i|^2|\tr(\nabla T)-T(\nabla \eta)|^2 dm + \int_\Omega{\bf u}_i\cdot (T(\tr(\nabla T), \nabla {\bf u}_i) - T( T(\nabla \eta), \nabla {\bf u}_i)))dm \nonumber \\
     &\leq C +\frac{1}{4}\int_{\Omega}|{\bf u}_i|^2\langle \tr(\nabla T), \tr(\nabla T) -2 T(\nabla \eta) \rangle dm +\int_\Omega{\bf u}_i\cdot T(\tr(\nabla T), \nabla {\bf u}_i) dm.
\end{align}
Replacing Inequality~\eqref{eqq2.24} into Lemma~\ref{lema2}, we complete the proof of Lemma~\ref{lemma2.3}.
\end{proof}

Now, we are in a position to give the proof of the two theorems of this paper. 

\section{Proof of Theorems~\ref{theorem1.1} and \ref{theorem1.2}}

\subsection{Proof of Theorem~\ref{theorem1.1}}
\begin{proof}
The proof is a consequence of Lemma~\ref{lemma2.3}. We begin by computing 
\begin{align*}
    \frac{1}{4}\int_{\Omega}|{\bf u}_i|^2\langle \tr(\nabla T), \tr(\nabla T) -2 T(\nabla \eta) \rangle dm =& \frac{1}{4} \int_{\Omega} |{\bf u}_i|^2 | \tr(\nabla T)|^2dm \\
    &- \frac{1}{2}\int_{\Omega}|{\bf u}_i|^2\langle \tr(\nabla T), T(\nabla \eta) \rangle dm.
\end{align*}
Since $T_0=\sup_{\Omega}|\tr(\nabla T)|$ and $\eta_0=\sup_{\Omega}|\nabla \eta|$, we have
\begin{align*}
    \frac{1}{4} \int_{\Omega} |{\bf u}_i|^2 | \tr(\nabla T)|^2dm \leq \frac{1}{4}T_0^2\int_\Omega |{\bf u}_i|^2 dm =\frac{1}{4} T_0^2,
\end{align*}
and using \eqref{T-norm} we get
\begin{align*}
    -\frac{1}{2}\int_{\Omega}|{\bf u}_i|^2\langle \tr(\nabla T), T(\nabla \eta) \rangle dm & \leq \frac{1}{2}\int_{\Omega}|{\bf u}_i|^2|\tr(\nabla T)| |T(\nabla \eta)| dm \nonumber\\
    &\leq \frac{\delta}{2}\int_{\Omega}|{\bf u}_i|^2|\tr(\nabla T)||\nabla \eta| dm \nonumber\\
    &\leq \frac{\delta}{2} T_0\eta_0.
\end{align*}
Then, 
\begin{eqnarray}\label{eq2.29}
    \frac{1}{4}\int_{\Omega}|{\bf u}_i|^2\langle \tr(\nabla T), \tr(\nabla T) -2 T(\nabla \eta) \rangle dm \leq \frac{1}{4} T_0^2 + \frac{\delta}{2} T_0\eta_0.
\end{eqnarray}
Furthermore,
\begin{align}\label{eqq31}
    \int_\Omega {\bf u}_i \cdot T(\tr(\nabla T), \nabla {\bf u}_i) dm &\leq \Big(\int_\Omega |{\bf u}_i|^2dm \Big)^{\frac{1}{2}} \Big(\int_\Omega |T(\tr(\nabla T), \nabla {\bf u}_i)|^2dm \Big)^{\frac{1}{2}} \nonumber\\
    & \leq T_0 \Big(\int_\Omega |T(\nabla {\bf u}_i)|^2 dm\Big)^\frac{1}{2} = T_0\|T(\nabla {\bf u}_i)\|.
\end{align}
Substituting \eqref{eq2.29} and \eqref{eqq31} into Lemma~\ref{lemma2.3}, we obtain
\begin{align*}
     \sum_{i=1}^k(\sigma_{k+1}-\sigma_i)^2 \leq & \frac{4(n\delta+\alpha)}{n^2\varepsilon^2} \sum_{i=1}^k(\sigma_{k+1}-\sigma_i)\Big\{\|T(\nabla {\bf u}_i)\|^2  + \frac{1}{4}T_0^2 + T_0\|T(\nabla {\bf u}_i)\|\\   & +\frac{\delta}{2} T_0\eta_0+ C \Big\}.
\end{align*}
 Moreover, from the proof of Lemma~\ref{lema2} and Lemma~\ref{lemma2.3}, we can see that 
\begin{align*}
   0 < \sum_{\beta=1}^n\Big\|T(\nabla x_\beta, \nabla {\bf u}_i) + \frac{1}{2}\mathscr{L}x_\beta {\bf u}_i\Big\|^2 \leq& \Big\{\|T(\nabla {\bf u}_i)\|^2  + \frac{1}{4}T_0^2 + T_0\|T(\nabla {\bf u}_i)\|\\ 
   &+\frac{\delta}{2} T_0\eta_0+ C
   \Big\}\\
  =&\Big(\|T(\nabla {\bf u}_i)\|  + \frac{1}{2}T_0 \Big)^2 +C_0,
\end{align*}
where $C_0=\frac{\delta}{2} T_0\eta_0+ C$ and $\{x_\beta\}_{\beta=1}^n$ are the canonical coordinate functions of $\mathbb{R}^n$. Thus, we get
\begin{align}\label{equation5-3}
     \sum_{i=1}^k(\sigma_{k+1}-\sigma_i)^2 \leq \frac{4(n\delta+\alpha)}{n^2\varepsilon^2} \sum_{i=1}^k(\sigma_{k+1}-\sigma_i)\Big\{\Big(\|T(\nabla {\bf u}_i)\|  + \frac{1}{2}T_0 \Big)^2 +C_0\Big\}.
\end{align}
From \eqref{problem1}, \eqref{parts} and \eqref{divformula} we obtain 
\begin{equation*}
    \sigma_i=\int_{\Omega}T(\nabla {\bf u}_i) \cdot \nabla {\bf u}_i dm + \alpha \|\dv_\eta{\bf u}_i\|^2.
\end{equation*}
Since there exist positive real numbers $\varepsilon$ and $\delta$ such that $\varepsilon I \leq T \leq \delta I$, from the previous inequality and \eqref{T-property}, we get
\begin{equation}\label{equation5-4}
    \|T(\nabla {\bf u}_i)\|^2 \leq \delta  \int_{\Omega}T(\nabla {\bf u}_i) \cdot \nabla {\bf u}_i dm = \delta(\sigma_i - \alpha \|\dv_\eta{\bf u}_i\|^2).
\end{equation}
Therefore, from \eqref{equation5-3} and \eqref{equation5-4} we obtain
\begin{equation*}
     \sum_{i=1}^k(\sigma_{k+1}-\sigma_i)^2 \leq \frac{4(n\delta+\alpha)}{n^2\varepsilon^2} \sum_{i=1}^k(\sigma_{k+1}-\sigma_i)\Big\{\Big[\sqrt{\delta}(\sigma_i - \alpha \|\dv_\eta{\bf u}_i\|^2)^{\frac{1}{2}}  + T_0 \Big]^2 +C_0\Big\},
\end{equation*}
which is enough to complete the proof.
\end{proof}

\subsection{Proof of Theorem~\ref{theorem1.2}}
\begin{proof}
Let $\{x_\beta \}_{\beta=1}^n$ be the standard coordinate functions of $\mathbb{R}^n$. Let us consider the matrix $D=(d_{ij})_{n \times n}$ where
\begin{equation*}
    d_{ij}:= \int_\Omega x_i{\bf u}_1\cdot {\bf u}_{j+1} dm.
\end{equation*}
Using the orthogonalization of Gram and Schmidt, we know that there exists an upper triangle matrix $R=(r_{ij})_{n \times n}$ and an orthogonal matrix $S=(s_{ij})_{n \times n}$ such that $R=SD$, namely
\begin{equation*}
    r_{ij}=\sum_{k=1}^ns_{ik}d_{kj}= \sum_{k=1}^n s_{ik} \int_\Omega x_k{\bf u}_1\cdot {\bf u}_{j+1} dm = \int_\Omega \Big( \sum_{k=1}^ns_{ik}x_k\Big){\bf u}_1\cdot {\bf u}_{j+1} dm = 0,
\end{equation*}
for $1 \leq j < i \leq n$. Putting $y_i=\sum_{k=1}^ns_{ik}x_k$, we have 
\begin{equation*}
    \int_\Omega y_i{\bf u}_1\cdot {\bf u}_{j+1} dm = 0   \quad \mbox{for} \quad 1 \leq j < i \leq n.
\end{equation*}
Let us denote by $a_i=\int_\Omega y_i |{\bf u}_1|^2dm$ to consider the vector-valued functions ${\bf w}_i$ given by
\begin{equation}\label{wi}
    {\bf w}_i = (y_i - a_i){\bf u}_1,
\end{equation}
so that 
\begin{equation*}
    {\bf w}_i|_{\partial \Omega}=0 \quad \mbox{and} \quad  \int_\Omega {\bf w}_i\cdot {\bf u}_{j+1} dm =0, \quad \mbox{for any} \quad j= 1, \ldots, i-1.
\end{equation*}
Then, we can take ${\bf v}={\bf w}_i$ in \eqref{rayleigh-1} and to use formula~\eqref{divformula} to obtain
\begin{equation}\label{equationn4.19}
    \sigma_{i+1} \|{\bf w}_i\|^2 \leq \int_\Omega ( -{\bf w}_i \cdot \mathscr{L}{\bf w}_i +  \alpha(\dv_\eta{\bf w}_i)^2)dm.
\end{equation}
 Using \eqref{equation2.2}  we get
\begin{align}\label{equationn4.20}
    &-\int_\Omega {\bf w}_i\cdot \mathscr{L}{\bf w}_idm =  -\int_\Omega {\bf w}_i\cdot[(y_i - a_i) \mathscr{L}{\bf u}_1 + {\bf u}_1 \mathscr{L}y_i + 2T(\nabla y_i, \nabla {\bf u}_1)]dm \nonumber\\
    &=  -\int_\Omega {\bf w}_i\cdot [(y_i - a_i)(-\sigma_1{\bf u}_1-\alpha\nabla \dv_\eta{\bf u}_1 ) + {\bf u}_1 \mathscr{L}y_i + 2T(\nabla y_i, \nabla {\bf u}_1)]dm \nonumber \\
    &= \sigma_1\|{\bf w}_i\|^2 +\alpha\int_\Omega (y_i-a_i){\bf w}_i \cdot \nabla \dv_\eta{\bf u}_1 dm - \int_\Omega {\bf w}_i\cdot ({\bf u}_1 \mathscr{L}y_i + 2T(\nabla y_i, \nabla {\bf u}_1))dm.
\end{align}
Using \eqref{property1} and \eqref{property2}, by a computation analogous to \eqref{eq4.7}, we obtain
\begin{align*}
   \alpha \int_\Omega (y_i-a_i){\bf w}_i \cdot \nabla \dv_\eta{\bf u}_1 dm =& - \alpha \int_\Omega (\dv_\eta{\bf w}_i)^2dm \\&- \alpha \int_\Omega (\nabla (\nabla y_i\cdot {\bf u}_1) + \dv_\eta {\bf u}_1\nabla y_i)\cdot {\bf w}_i dm.
\end{align*}
Substituting the previous equality into \eqref{equationn4.20}, we get
\begin{align}\label{equationn4.21}
    \int_\Omega(-{\bf w}_i\cdot \mathscr{L}{\bf w}_i+\alpha(\dv_\eta{\bf w}_i)^2)dm=&\sigma_1\|{\bf w}_i\|^2 - \alpha \int_\Omega (\nabla (\nabla y_i\cdot {\bf u}_1) + \dv_\eta {\bf u}_1\nabla y_i)\cdot {\bf w}_idm \nonumber\\
    &- \int_\Omega {\bf w}_i\cdot ({\bf u}_1 \mathscr{L}y_i + 2T(\nabla y_i, \nabla {\bf u}_1))dm.
\end{align}
Replacing \eqref{equationn4.21} into \eqref{equationn4.19}, we have
\begin{align}\label{equationn4.22}
    (\sigma_{i+1} - \sigma_1)\|{\bf w}_i\|^2 \leq & 
    -\int_\Omega {\bf w}_i\cdot ({\bf u}_1 \mathscr{L}y_i + 2T(\nabla y_i, \nabla {\bf u}_1))dm \nonumber\\
    &-\alpha \int_\Omega {\bf w}_i \cdot (\nabla (\nabla y_i\cdot {\bf u}_1) + \dv_\eta {\bf u}_1\nabla y_i) dm.
\end{align}
By a straightforward computation, we have, from \eqref{property1}, \eqref{property2}, \eqref{parts} and \eqref{wi},
\begin{align}\label{equationn4.23}
    - \int_\Omega {\bf w}_i\cdot ({\bf u}_1 \mathscr{L}y_i + 2T(\nabla y_i, \nabla {\bf u}_1))dm =  \int_\Omega |{\bf u}_1|^2T(\nabla y_i, \nabla y_i) dm.
\end{align}
\begin{align}\label{equationn4.24}
    - \alpha \int_\Omega {\bf w}_i \cdot (\nabla (\nabla y_i\cdot {\bf u}_1) + \dv_\eta {\bf u}_1\nabla y_i)dm = \alpha\int_\Omega |\nabla y_i \cdot {\bf u}_1|^2dm.
\end{align}
Therefore, substituting \eqref{equationn4.23} and \eqref{equationn4.24} into \eqref{equationn4.22} we obtain
\begin{align}\label{equationn4.25}
    (\sigma_{i+1} - \sigma_1)\|{\bf w}_i\|^2 \leq \int_\Omega |{\bf u}_1|^2T(\nabla y_i, \nabla y_i) dm + \alpha\int_\Omega |\nabla y_i \cdot {\bf u}_1|^2dm.
\end{align}
From \eqref{equationn4.23}, for any constant $B >0$,  we infer
\begin{align*}
    (\sigma_{i+1} - \sigma_1)&\int_\Omega |{\bf u}_1|^2T(\nabla y_i, \nabla y_i) dm \\
    =&  (\sigma_{i+1} - \sigma_1)\Big\{ - 2\int_\Omega {\bf w}_i\cdot \Big(\frac{1}{2}{\bf u}_1 \mathscr{L}y_i + T(\nabla y_i, \nabla {\bf u}_1)\Big)dm\Big\} \\
    \leq& 2 (\sigma_{i+1} - \sigma_1)\|{\bf w}_i\|\Big\| \frac{1}{2}{\bf u}_1 \mathscr{L}y_i + T(\nabla y_i, \nabla {\bf u}_1) \Big\| \\
    \leq& B(\sigma_{i+1} - \sigma_1)^2\|{\bf w}_i\|^2 + \frac{1}{B}\Big\| \frac{1}{2}{\bf u}_1 \mathscr{L}y_i + T(\nabla y_i, \nabla {\bf u}_1) \Big\|^2,
\end{align*}
hence using \eqref{equationn4.25} and the previous inequality we get
\begin{align}\label{equationn4.26}
    (\sigma_{i+1} - \sigma_1)&\int_\Omega |{\bf u}_1|^2T(\nabla y_i, \nabla y_i) dm \nonumber\\
\leq& B(\sigma_{i+1} - \sigma_1) \Big( \int_\Omega |{\bf u}_1|^2T(\nabla y_i, \nabla y_i) dm + \alpha\int_\Omega |\nabla y_i \cdot {\bf u}_1|^2dm \Big) \nonumber\\
    & +\frac{1}{B}\Big\| \frac{1}{2}{\bf u}_1 \mathscr{L}y_i + T(\nabla y_i, \nabla {\bf u}_1) \Big\|^2.
\end{align}
Summing over $i$ from $1$ to $n$ in \eqref{equationn4.26}, we conclude that
\begin{align}\label{equationn4.27}
\sum_{i=1}^n&(\sigma_{i+1} - \sigma_1)(1-B)\int_\Omega |{\bf u}_1|^2T(\nabla y_i, \nabla y_i) dm \nonumber\\ &\leq B\alpha\sum_{i=1}^n(\sigma_{i+1} - \sigma_1)\int_\Omega |\nabla y_i \cdot {\bf u}_1|^2dm
+\frac{1}{B}\sum_{i=1}^n\Big\| \frac{1}{2}{\bf u}_1 \mathscr{L}y_i + T(\nabla y_i, \nabla {\bf u}_1) \Big\|^2.
\end{align}
From the definition of $y_i$ and the fact that $S$ is an orthogonal matrix, we know that $\{y_i\}_{i=1}^n$ are also the coordinate functions in $\mathbb{R}^n$. Then, as in the proof of Theorem~\ref{theorem1.1}, we can also get
\begin{equation*}
   0 < \sum_{i=1}^n\Big\| \frac{1}{2}{\bf u}_1 \mathscr{L}y_i + T(\nabla y_i, \nabla {\bf u}_1) \Big\|^2 \leq \Big(\| T(\nabla {\bf u}_1) \| + \frac{1}{2}T_0\Big)^2 + C_0,
\end{equation*}
where $C_0$ is given by Eq.~\eqref{C_0}. Using \eqref{equationn4.27} and $\varepsilon I \leq T \leq \delta I$, we obtain
\begin{align}\label{equationn4.28}
    \sum_{i=1}^n(\sigma_{i+1} - \sigma_1)(\varepsilon - B(\delta + \alpha)) \leq \frac{1}{B}\Big\{  (\| T(\nabla {\bf u}_1) \| + \frac{1}{2}T_0)^2+C_0 \Big\}.
\end{align}
Since $B$ is an arbitrary positive constant, we can take 
\begin{equation*}
    B=\Bigg\{\frac{(\| T(\nabla {\bf u}_1) \| + \frac{1}{2}T_0)^2 + C_0}{(\delta+ \alpha)\sum_{i=1}^n(\sigma_{i+1} - \sigma_1)}\Bigg\}^{\frac{1}{2}}
\end{equation*}
into \eqref{equationn4.28} and therefore we get 
\begin{equation}\label{equation5.16}
    \sum_{i=1}^n(\sigma_{i+1} - \sigma_1) \leq \frac{4(\delta+\alpha)}{\varepsilon^2}\Big\{  (\| T(\nabla {\bf u}_1) \| + \frac{1}{2}T_0)^2 + C_0\Big\}.
\end{equation}
We can take $i = 1$ in inequality \eqref{equation5-4} and replace in \eqref{equation5.16} to obtain Theorem~\ref{theorem1.2}.
\end{proof}

\section{Divergence-free tensors case}\label{Sec-DCYOp}

This section is a generalization of some results of Section~\ref{section}. Here, we are assuming the tensor $T$ to be {\it divergence-free}, i.e., $\dv T =0$. Divergence-free tensors often arise from physical facts. We can find some of them in fluid dynamics, for instance, in the study of: compressible gas; rarefied gas; steady/self-similar flows and relativistic gas dynamics, see e.g. Serre~\cite{Serre}. We highlight that Serre's work deals with divergence-free positive definite symmetric tensors and fluid dynamics.

For divergence-free tensors, the operator $\mathscr{L}$ can be decomposed as follows
\begin{equation}\label{DCYOp}
    \mathscr{L}f = \square f - \langle\nabla \eta, T(\nabla f)\rangle,
\end{equation}
where $\square$ is the operator introduced by Cheng and Yau~\cite{ChengYau}, namely:
\begin{equation*}
    \square f = \tr{(\nabla^2f \circ T)}=\langle \nabla^2 f, T\rangle.
\end{equation*}

Cheng-Yau operator arise from the study of complete hypersurfaces of constant scalar curvature in space forms. For more details, the reader can be consult Gomes and Miranda~\cite{GomesMiranda}.

Eq.~\eqref{DCYOp} is a first-order perturbation of the  Cheng-Yau operator, and it defines a {\it drifted Cheng-Yau operator} which we denote by $\square_\eta$ with a drifting function $\eta$.

We now turn our attention to the main problem of this paper. Since $T$ is divergence-free, the coupled system of second-order elliptic differential equations \eqref{problem1} becomes
\begin{equation}\label{problem1-3}
    \left\{\begin{array}{ccccc} 
    \square_\eta  {\bf u} + \alpha \nabla(\dv_\eta {\bf u}) &=& -\sigma {\bf u} & \mbox{in } & \Omega,\\
     {\bf u}&=&0 & \mbox{on} & \partial\Omega, 
    \end{array} 
    \right.
\end{equation}
where ${\bf u}=(u^1, u^2, \ldots, u^n)$ is a vector-valued function from $\Omega$ to $\mathbb{R}^n$, the constant $\alpha$ is non-negative and $\square_\eta{\bf u}:=(\square_\eta u^1, \square_\eta u^2, \ldots,\square_\eta  u^n)$. Moreover, we have $\tr{(\nabla T)=0}$, because $T$ is divergence-free. Thus, the constant $C_0$ in \eqref{C_0} becomes
\begin{equation*}
C_0=\sup_{\Omega}{\Big \{}\frac{1}{2}\dv (T^2(\nabla \eta)) - \frac{1}{4}|T(\nabla \eta)|^2{\Big \}}.
\end{equation*}
Hence, from Theorems~\ref{theorem1.1} and \ref{theorem1.2} we immediately obtain the next two corollaries.
\begin{cor}\label{corollary-6.1}
Let $\Omega \subset \mathbb{R}^n$ be a  bounded domain, and ${\bf u}_i$ be a normalized eigenfunction corresponding to $i$-th eigenvalue $\sigma_i$ of Problem~\ref{problem1-3}. For any positive integer $k$, we get
\begin{equation*}
    \sum_{i=1}^k(\sigma_{k+1}-\sigma_i)^2 \leq \frac{4\delta(n\delta+\alpha)}{n^2\varepsilon^2}\sum_{i=1}^k(\sigma_{k+1}-\sigma_i)\Big(\sigma_i - \alpha \|\dv_\eta{\bf u}_i\|^2 + \frac{C_0}{\delta}\Big).
\end{equation*}
\end{cor} 
\begin{cor}\label{corollary-6.2}
Let $\Omega \subset \mathbb{R}^n$ be a bounded domain, $\sigma_i$ be the $i$-th eigenvalue of Problem~\ref{problem1-3}, for $i=1,\ldots,n$, and ${\bf u}_1$ be a normalized eigenfunction corresponding to the first eigenvalue. Then, we get
\begin{equation*}
    \sum_{i=1}^n(\sigma_{i+1} - \sigma_1) \leq \frac{4\delta(\delta+\alpha)}{\varepsilon^2}(\sigma_1 +D_1),
\end{equation*}
where $D_1= -\alpha\|\dv_\eta {\bf u}_1\|^2 + \frac{C_0}{\delta}$.
\end{cor}

Now, from Corollary~\ref{corollary-6.1} and following the same steps of the proof of Corollary~\ref{cor_sharpgap}, we obtain the estimates.
\begin{cor}\label{corollary-6.3}
Under the same setup as in Corollary~\ref{corollary-6.1}, and by defining $D_0=-\alpha \min_{j=1, \ldots, k}\|\dv_\eta {\bf u}_j\|^2+\frac{C_0}{\delta}$, we have
\begin{align*}
    \sigma_{k+1} + D_0 \leq& \Big(1+\frac{2\delta(n\delta+\alpha)}{\varepsilon^2n^2}\Big)\frac{1}{k}\sum_{i=1}^k(\sigma_i + D_0) + \Big[\Big( \frac{2\delta(n\delta+\alpha)}{\varepsilon^2n^2}\frac{1}{k}\sum_{i=1}^k(\sigma_i + D_0)\Big)^2 \nonumber \\
   &-\Big(1+\frac{4\delta(n\delta+\alpha)}{\varepsilon^2n^2}\Big) \frac{1}{k}\sum_{j=1}^k\Big(\sigma_j-\frac{1}{k}\sum_{i=1}^k\sigma_i\Big)^2 \Big]^{\frac{1}{2}}
\end{align*}
and
\begin{eqnarray*}
\nonumber\sigma_{k+1} -\sigma_k &\leq& 2\Big[\Big( \frac{2\delta(n\delta+\alpha)}{\varepsilon^2n^2}\frac{1}{k}\sum_{i=1}^k(\sigma_i + D_0)\Big)^2\\
&&-\Big(1+\frac{4\delta(n\delta+\alpha)}{\varepsilon^2n^2}\Big) \frac{1}{k}\sum_{j=1}^k\Big(\sigma_j-\frac{1}{k}\sum_{i=1}^k\sigma_i\Big)^2 \Big]^{\frac{1}{2}}.
\end{eqnarray*}
\end{cor}

Again from Corollary~\ref{corollary-6.1} and by applying the recursion formula of Cheng and Yang~\cite{ChengYang2}, we obtain the next corollary.
\begin{cor}\label{corollary3.3-CY}
Under the same setup as in Corollary~\ref{corollary-6.3}, we have
\begin{equation*}
    \sigma_{k+1} + D_0 \leq \Big(1+\frac{4\delta(\delta n+\alpha)}{\varepsilon^2n^2}\Big)k^{\frac{2\delta( n\delta+\alpha)}{\varepsilon^2 n^2}}(\sigma_1 + D_0).
\end{equation*}
\end{cor}

\section*{Acknowledgements} 
The authors would like to express their sincere thanks to Chang Yu Xia and Dragomir Mitkov Tsonev for useful comments, discussions and constant encouragement. The first author has been partially supported by Coordenação de Aperfeiçoamento de Pessoal de Nível Superior (CAPES) in conjunction with Fundação Rondônia de Amparo ao Desenvolvimento das Ações Científicas e Tecnológicas e à Pesquisa do Estado de Rondônia (FAPERO). The second author has been partially supported by Conselho Nacional de Desenvolvimento Científico e Tecnológico (CNPq), of the Ministry of Science, Technology and Innovation of Brazil.

\end{document}